\documentclass{amsart}
\usepackage{graphicx, amssymb, mathrsfs}
\usepackage[all, cmtip]{xy}

\newtheorem{thm}{Theorem}[section]
\newtheorem{cor}[thm]{Corollary}
\newtheorem{lem}[thm]{Lemma}
\newtheorem{prop}[thm]{Proposition}
\newtheorem{conj}[thm]{Conjecture}
\theoremstyle{definition}
\newtheorem{defn}[thm]{Definition}
\theoremstyle{remark}
\newtheorem{rem}[thm]{Remark}
\newtheorem{exam}[thm]{Example}
\numberwithin{equation}{section}
\newcommand{\norm}[1]{\left\Vert#1\right\Vert}
\newcommand{\abs}[1]{\left\vert#1\right\vert}
\newcommand{\set}[1]{\left\{#1\right\}}
\newcommand{\Real}{\mathbb R}
\newcommand{\eps}{\varepsilon}

\newcommand{\cali}[1]{\mathscr{#1}}
\newcommand{\cplx}{\mathbb{C}}
\newcommand{\proj}{\mathbb{P}}

\newcommand{\V}{\cali{V}}
\newcommand{\C}{\cali{C}}
\newcommand{\Aut}{\mathrm{Aut}}
\newcommand{\agl}[1]{\langle{#1}\rangle}

\newcommand{\dist}{\mathrm{dist}}
\newcommand{\dsh}{\mathrm{DSH}}
\newcommand{\id}{\mathrm{id}}
\newcommand{\supp}{\mathrm{supp}}
\newcommand{\multi}{\mathrm{multi}}
\newcommand{\codim}{\mathrm{codim}}
\begin{document}

\title[Equidistribution for holomorphic endomorphisms of $\proj^k$]{Equidistribution in higher codimension \protect\\ for holomorphic endomorphisms of $\proj^k$}%
\author{Taeyong Ahn}%

\thanks{The research of the author was supported in part by SRC-GaiA (Center for Geometry and its Applications), the Grant 2011-0030044 from The Ministry of Education, The Republic of Korea.}
\address{Center for Geometry and its Applications and Department of Mathematics, 
POSTECH, Pohang City 790-784, The Republic of Korea}%
\email{triumph@postech.ac.kr}%


\date{March 12, 2014}%
\begin{abstract}
In this paper, we discuss the equidistribution phenomena for holomorphic endomorphisms over $\proj^k$ in the case of bidegree $(p,p)$ with $1\leq p\leq k$, in particular, $1<p<k$. We prove that if $f:\proj^k\to\proj^k$ is a holomorphic endomorphism of degree $d\geq 2$ and $T^p$ denotes the Green $(p,p)$-current associated with $f$, then there exists a proper invariant analytic subset $E$ for $f$ such that $d^{-pn}(f^n)^*(S)$ converges to $T^p$ exponentially fast in the current sense for every positive closed $(p,p)$-current $S$ of mass $1$ which is smooth on $E$.
\end{abstract}
\maketitle

\section{Introduction}
In 1965, Hans Brolin proved the following theorem about the distribution of preimages of points for polynomial maps in one variable in ~\cite{HansBrolin}.
\begin{thm}\label{thm:brolin}
Let $f(z)=z^d+...$ be a given polynomial of degree $d\geq 2$. Then, there exists a subset $\cali{E}\subset\cplx$ such that $\sharp\cali{E}\leq 1$ such that if $a\in\cplx\setminus\cali{E}$, then
\begin{displaymath}
\frac{1}{d^n}\sum_{f^n(z)=a}\delta_z\to\mu\textrm{ as }n\to\infty
\end{displaymath}
where $\mu$ is a harmonic measure on the filled Julia set of $f$. The limit is independent of the choice of $a\in\cplx\setminus\cali{E}$. The exceptional set $\cali{E}=\emptyset$ unless $f$ is affinely conjugate to $z\to z^d$. In this case, the set $\cali{E}=\set{0}$ is totally invariant.
\end{thm}

Such convergence towards a unique invariant probability measure or current is called \emph{equidistribution}.
In the study of the dynamics of $f$, ergodic theory plays an important role. Here, it is crucial to have a dynamically interesting invariant probability measure for $f$. Equidistribution provides a way to construct such an invariant probability measure. 
Also, the invariant probability measure $\mu$ for $f$ is useful in studying the Julia set of $f$.

Theorem ~\ref{thm:brolin} generalizes to more general cases. 
Lyubich\cite{Lyubich}, and Freire-Lopes-Ma\~{n}\'{e}\cite{FLM} independently studied the case of the rational maps of the Riemann sphere $\proj^1$ with $\sharp\cali{E}\leq 2$.

Many authors contributed to the study of the higher dimensional case. The Dirac measure $\delta_z$ generalizes to a positive closed current. Dinh-Sibony\cite{equidistributionspeed} completed the measure case. See also Briend-Duval\cite{BD} and Forn{\ae}ss-Sibony\cite{cplxpotentialtheory}. The case of bidegree $(1,1)$ has been also well investigated. 
The case of $\proj^2$ was finished by Favre-Jonsson\cite{brolinthm}\cite{eigenvaluations}(see also \cite{Guedj}). In the general higher dimensional case, see Dinh-Sibony\cite{equidist_holo}, Forn{\ae}ss-Sibony\cite{cplxdynII}, Guedj\cite{Guedj}, Russakovskii-Shiffman\cite{RS} and Sibony\cite{sibony}. For recent developments, see Parra\cite{Parra} and Taflin\cite{Taflin}.\\ 

However, the intermediate bidegree case, i.e., the case of bidegree $(p,p)$ with $1<p<$ the dimension of the space, does not seem to have been investigated beyond the following theorem (for the definition of $\cali{C}_p$ and the Green $(p, p)$-current $T^p$ associated with $f$, see Section \ref{sec:currents}).

\begin{thm}[See Theorem 5.4.4 in \cite{superpotentials}]\label{thm:544}
Let $\cali{H}_d(\proj^k)$ denote the set of holomorphic endomorphisms of degree $d\geq 2$ on $\proj^k$.
There is a Zariski dense open set $\cali{H}_d^*(\proj^k)$ in $\cali{H}_d(\proj^k)$ such that, if $f$ is in $\cali{H}_d^*(\proj^k)$, then $d^{-pn}(f^n)^*(S)$ converges to $T^p$ uniformly with respect to $S\in\cali{C}_p$. In particular, for $f$ in $\cali{H}_d^*(\proj^k)$, $T^p$ is the unique current in $\cali{C}_p$ which is $f^*$-invariant.
\end{thm}

The purpose of this paper is to generalize Theorem ~\ref{thm:544} by proving the following theorem.
\begin{thm}\label{thm:mainthm}
Let $f:\proj^k\to\proj^k$ be a holomorphic endomorphism of degree $d\geq 2$. Let $T^p$ denote the Green $(p, p)$-current associated with $f$ on $\proj^k$. Then, there is a proper invariant analytic subset $E$ for $f$ such that 
$d^{-pn}(f^n)^*(S)$ converges to $T^p$
exponentially fast in the current sense for every $S\in\cali{C}_p$ smooth on $E$.
\end{thm}
Here, for $S$ in $\cali{C}_p$, we say that $S$ is smooth on $E$ if there exists a neighborhood of $E$ on which the restriction of $S$ can be represented by a smooth $(p, p)$-form.
 

The major difficulty of Theorem ~\ref{thm:mainthm} is finding good localizations of the set $E$ as in Lemma \ref{lem:cut-off}. The set $E$ is obtained from the work of Dinh in ~\cite{analyticmulticocycle}. Roughly speaking, $E$ can be understood as a high multiplicity set invariant under $f$. In order to handle this difficulty, we approximate a quasi-potential of the current of integration on the hypersurface $V$ of the critical values of $f$ in terms of multiplicities and the distances to $V$ and $E$ in Lemma ~\ref{lem:lojasiewicz}. We use Lojasiewicz type inequalities for this approximation. There, we adapt the idea used in \cite{equidistributionspeed}. Lemma ~\ref{lem:lojasiewicz} intuitively means that if a point near $V$ is relatively far from a high-multiplicity set, then the effect of the high multiplicity set is not serious. It is reflected in the coefficient $\delta$ of $\log \dist(\cdot, V)$ in the inequality of Lemma ~\ref{lem:lojasiewicz}. Note that this difficulty does not appear in Theorem \ref{thm:544} since a global multiplicity condition is assumed in Lemma 5.4.5 in ~\cite{superpotentials}.

In addition to the major difficulty, we also have two obstacles in the intermediate bidegree case: lack of good potential/pluripotential theory and lack of good singularity theory in the higher codimensional case, such as the concept of the Lelong number for the case of bidegree $(1, 1)$. Indeed, the main ingredients in the case of bidegree $(1,1)$ are pluripotential theory and the Lelong number.
In Theorem ~\ref{thm:544}, the first obstacle was resolved by super-potentials introduced in \cite{superpotentials} by Dinh-Sibony. However, for the second obstacle, we still do not have a successful theory in the intermediate bidegree case by far. Instead, Dinh-Sibony approximate higher bidegree objects by objects of bidegree $(1,1)$. We can understand this approximation by $f_*(\omega^p)\leq (f_*\omega)^p$, where $\omega$ is the standard Fubini-Study form. For these two difficulties, we followed the strategy of Theorem ~\ref{thm:544} in ~\cite{superpotentials}: super-potentials and the current inequality $f_*(\omega^p)\leq (f_*\omega)^p$.

As a corollary, we obtain
\begin{cor}\label{cor:544tomain}
Theorem ~\ref{thm:mainthm} implies Theorem ~\ref{thm:544}. In particular, the set $E$ is generically empty.
\end{cor}


\begin{rem}
The following conjecture was posed by Dinh-Sibony for the intermediate bidegree case.
\begin{conj}[See Conjecture 1.4 in ~\cite{equidist_holo}]\label{conj:dinh-sibony}
Let $f:\proj^k\to\proj^k$ be a holomorphic endomorphism of degree $d\geq 2$ and $T$ its Green $(1, 1)$-current. Then $d^{-pn}(f^n)^*[H]$ converges to $sT^p$ for every analytic subset $H$ of $\proj^k$ of pure dimension $p$ and of degree $s$ which is generic. Here, $H$ is generic if either $H\cap \mathcal{E}=\emptyset$ or $\codim H\cap \mathcal{E}=p+\codim \mathcal{E}$ for any irreducible component $\mathcal{E}$ of every totally invariant analytic subset of $\proj^k$ and $[H]$ denotes the current of integration on $H$.
\end{conj}
Notice that Theorem ~\ref{thm:mainthm} gives a partial answer to Conjecture ~\ref{conj:dinh-sibony}.
\end{rem}

This paper is organized as follows. In Section \ref{sec:analyticcocycle}, we find the desired invariant analytic subset $E$ for a given holomorphic endomorphism $f:\proj^k\to\proj^k$. From Section \ref{sec:lineq} through Section \ref{sec:superpotentials}, we summarize preliminaries. In Section \ref{sec:settings}, we prove the main theorem. From Section \ref{sec:1} to Section \ref{sec:4}, we complete the details of the computations. In the last section, we give examples where Theorem \ref{thm:mainthm} is applicable but Theorem \ref{thm:544} is not and we put some remarks.

\subsection*{Notation} In this paper, we consider an arbitrary holomorphic endomorphism $f:\proj^k\to\proj^k$ of degree $d\geq 2$. We use $f$ only for this endomorphism. For a general function, we will use $g$. We denote by $\Phi_n$ and $\Psi_n$ the hypersurfaces of the critical points and the critical values of $f^n$, respectively.

Since $\Aut(\proj^k)$ is a complex Lie group of dimension $k^2+2k$, we will work with a fixed local holomorphic coordinate chart in a neighborhood of $\id\in\Aut(\proj^k)$. We denote by $\zeta$ a coordinate system over the chart such that $\zeta=0$ at $\id$ and by $\tau_\zeta$ its corresponding automorphism in $\Aut(\proj^k)$. We choose a norm $\norm{\zeta}_A$ on the coordinate system $\zeta$ such that the norm $\norm{\cdot}_A$ is invariant under the involution $\tau\to\tau^{-1}$ and such that $\set{\norm{\zeta}_A<1}$ lies inside the coordinate system. We fix a smooth probability measure $\rho$ for the coordinate system $\set{\zeta}$ with compact support in $\set{\norm{\zeta}_A<1}$ such that $\rho$ is radial and decreasing as $\norm{\zeta}_A$ increases. In particular, $\rho$ is preserved under the involution $\tau\to\tau^{-1}$.

We use $\omega$ for the standard Fubini-Study form and the distance $\dist(\cdot, \cdot)$ on $\proj^k$ is measured with respect to the standard Fubini-Study metric unless stated otherwise. We use $\dist_{\mathrm{euc}}(\cdot, \cdot)$ for the Euclidean distance. We use $\norm{\cdot}$ for the standard Euclidean norm and for the mass of a current (cf. Section \ref{sec:currents}), but the meaning becomes clear from the context. We denote by $\norm{\cdot}_{\cali{L}^p}$ (or $\norm{\cdot}_{\cali{L}^p(U)}$) the $\cali{L}^p$-norm on $\proj^k$ (or on the set $U\subseteq \proj^k$, resp.), by $\norm{\cdot}_\infty$ (or $\norm{\cdot}_{\infty, U}$) the $\sup$-norm on $\proj^k$ (or on the set $U\subseteq \proj^k$, resp.) and by $\norm{\cdot}_{\cali{C}^\alpha}$ (or $\norm{\cdot}_{\cali{C}^\alpha(U)}$) the $\cali{C}^\alpha$-norm on $\proj^k$ (or on the set $U\subseteq \proj^k$, resp.). By the $\cali{C}^\alpha$-norm (or the sup-norm) of a form, we mean the sum of the $\cali{C}^\alpha$-norms (or the sup-norms, resp.) of its coefficients with respect to a fixed finite atlas of $\proj^k$. The $\epsilon$-neighborhood $U_\epsilon$ of a set $U$ in $\proj^k$ for $\epsilon>0$ is defined by $\set{x\in\proj^k:\dist(x, U)<\epsilon}$.

We denote by $\nu(x, R)$ the Lelong number of a positive closed $(1,1)$-current $R$ at $x$ and by $[V]$ the current of integration over $V$ where $V$ is an analytic subset of $\proj^k$. We abbreviates ``(quasi-)plurisubharmonic'' to ``(q-)psh'', respectively.

We will use two different types of multiplicities. By the multiplicity of a function $g$ at a point $x$, we mean the number of the preimages near $x$ of $y$ under $g$ for a generic $y$ close to $g(x)$. The multiplicity of an analytic subvariety (or subset) $V$ of dimension $l$ in $\proj^k$ at a point $P$, denoted by $\multi_P(V)$, is taken to be the number of sheets in the projection, in a small coordinate polydisc in $\cplx^k\subset\proj^k$ around $P$, of $V$ onto a generic $l$-dimensional polydisc. For the details about the projection, see ~\cite{cplxanalvar}.

\subsection*{Acknowledgement}
The author would like to thank John Erik Forn{\ae}ss for introducing this problem and his support and advice. The author would like to thank Tien-Cuong Dinh and Nessim Sibony for their advice and comments. The author would like to thank the referee for his careful reading. His comments and suggestions were helpful in improving and clarifying the arguments.

\section{Analytic (Sub-)Multiplicative Cocycles}\label{sec:analyticcocycle}
In this section, we consider two different types of multiplicities related to $f^n$ as $n\to\infty$. As a tool, we use the concept of the analytic (sub-)multiplicative cocycle. It was first introduced by Favre\cite{favre0}\cite{favre1} and further studied by Dinh\cite{analyticmulticocycle} and Gignac\cite{Gignac}.

Let $X$ be an irreducible compact complex space of dimension $k$, not necessarily smooth. Let $g:X\to X$ be an open holomorphic map. 

\begin{defn}[See Definition 1.1 in ~\cite{analyticmulticocycle}]\label{defn:acocycle}
A sequence $\set{\kappa_n}$ of functions $\kappa_n:X\to (0, \infty)$ for $n\geq 0$ is said to be an analytic sub-multiplicative (resp., multiplicative) cocycle (with respect to $g$), if for all $m, n\geq 0$, for all $x\in X$,
\begin{enumerate}
\item $\kappa_n$ is upper-semicontinuous (usc) with respect to the Zariski topology on $X$ and $\kappa_n\geq {c_\kappa}^n$ for some constant $c_\kappa>0$, and
\item $\kappa_{n+m}(x)\leq\kappa_n(x)\cdot\kappa_m(g^n(x))$ (resp., =).
\end{enumerate}
\end{defn}

\begin{defn}[See Introduction in ~\cite{analyticmulticocycle}]
$\kappa_{-n}(x)=\max_{y\in g^{-n}(x)}\kappa_n(y)$.
\end{defn}

Observe that $\kappa_{-n}$ is usc in the Zariski sense. The key theorem to use in this section is the following:

\begin{thm}[See Theorem 1.2 in ~\cite{analyticmulticocycle}]\label{thm:kappa-}
The sequence $\set{(\kappa_{-n})^{1/n}}$ converges to a function $\kappa_-$ defined over $X$ with the following properties: for all $\delta>\inf_X\kappa_-$, $\set{\kappa_-\geq\delta}$ is a proper analytic subset of $X$, invariant under $g$ and contained in the orbit of $\set{\kappa_n\geq\delta^n}$ for all $n\geq 0$. In particular, $\kappa_-$ is usc in the Zariski sense.
\end{thm}


For each $n\in\mathbb{N}$, define $\mu_n(x)$ to be the local multiplicity of $f^n$ at $x\in\proj^k$. Then, 
$\set{\mu_n}$ is an analytic multiplicative cocycle with respect to $f$.
We consider the following two types of multiplicities related to $f^n$:
\subsection{The function $\mu_{-n}$} By Theorem \ref{thm:kappa-}, the limit function $\mu_-$ exists for $\set{\mu_n}$ with $\min_{\proj^k}\mu_-=1$ and $1\leq \mu_-\leq d$.
Let $\lambda$ be arbitrarily given such that $1<\lambda<d$. We define $E_\lambda:=\set{\mu_-\geq d\lambda^{-1}}$. The set $E_\lambda$ is a proper invariant analytic subset for $f$. 
We obtain:
\begin{lem}[See Theorem 1.2  in \cite{analyticmulticocycle}] \label{lem:multiplicitycondition1}
There are a sufficiently large $n_\lambda\in\mathbb{N}$ and some $\delta_\lambda$ with $1<\delta_\lambda<d\lambda^{-1}$ such that $\mu_{-n_\lambda m}<{\delta_\lambda}^{n_\lambda m}$ on $\proj^k\setminus E_\lambda$ for all $m\in \mathbb{N}$.
\end{lem}
\subsection{The multiplicity of $\Psi_n$ as an analytic subset} 
We modify the argument in Section 3 of Favre-Jonsson\cite{brolinthm}.
For notational convenience, let $\iota(x, g):=\nu(x, dd^c\log\abs{g})$. For each $n\in\mathbb{N}$, we define $\mu'_n(x):=2k-1+2\iota(x, J_{f^n})$ on $\proj^k$, where $J_{f^n}$ denotes the Jacobian determinant of $f^n$. Then, $\iota(x, J_{f^{m+n}})=\iota(x, J_{f^n})+\iota(x, J_{f^m}\circ f^n)$ for any $x\in\proj^k$ and any $m, n\in\mathbb{N}$. The proposition below
 implies that $\set{\mu'_n}$ is an analytic submultiplicative cocycle with respect to $f$ (see Section 3 in \cite{brolinthm}).
\begin{prop}[See Remark 3 in ~\cite{favre}; for a sharper version, see also ~\cite{Parra1}]
For any $x\in\proj^k$ and any $m, n\geq 0$, the following inequality
$$\iota(x, J_{f^m}\circ f^n)\leq(2k-1+2\iota(x, J_{f^n}))\cdot\iota(f^n(x), J_{f^m})$$
holds.
\end{prop}

Hence, by Theorem \ref{thm:kappa-}, the limit function $\mu'_-$ exists for $\set{\mu'_n}$.
 We have $\min_{\proj^k}\mu'_-=1$ and $1\leq \mu'_-\leq d$. 
Let $\lambda'$ be arbitrarily given so that $1<\lambda'<d$. Let $E'_{\lambda'}:=\set{\mu'_-\geq d(\lambda')^{-1}}$. The set $E'_{\lambda'}$ is a proper invariant analytic subset for $f$. Then, as previously done for $\set{\mu_n}$, there are a sufficiently large $n'_{\lambda'}\in\mathbb{N}$ and some $\delta'_{\lambda'}$ with $1<\delta'_{\lambda'}<d(\lambda')^{-1}$ such that $\mu'_{-n'_{\lambda'}m}<(\delta'_{\lambda'})^{n'_{\lambda'}m}$ on $\proj^k\setminus E_{\lambda'}$ for all $m\in \mathbb{N}$. In particular, for all $m\in\mathbb{N}$, for each $x\in\Psi_{n'_{\lambda'}m}\setminus E'_{\lambda'}$, we have that
\begin{eqnarray}
\label{ineq:maxmulticrit}&&\max_{y\in \Phi_{n'_{\lambda'} m}\cap f^{-n'_{\lambda'} m}(x)}\nu(y, [\Phi_{n'_{\lambda'} m}])\\
&&\nonumber\leq \max_{y\in \Phi_{n'_{\lambda'} m}\cap f^{-n'_{\lambda'} m}(x)}\nu(y, dd^c\log\abs{J_{f^{n'_{\lambda'}m}}})\leq(\delta'_{\lambda'})^{n'_{\lambda'} m}<\left(\frac{d}{\lambda'}\right)^{n'_{\lambda'} m}.
\end{eqnarray}

From this inequality, we obtain a main lemma of this section:
\begin{lem}\label{lem:multiplicityconditions}
Assume that $\lambda$ and $\lambda'$ are arbitrarily given such that $1<\lambda, \lambda'<d$. Let $E_\lambda$, $n_\lambda$ and $E'_{\lambda'}$, $n'_{\lambda'}$ be defined as above in the discussion. Let $E:=E_\lambda\cup E'_{\lambda'}$ and $N_E:=n_\lambda n'_{\lambda'}$. Then, for any positive integer $j\in\mathbb{N}$,
\begin{enumerate}
\item $E$ is invariant under $f$,
\item $\mu_{-jN_E}(x)<(d\lambda^{-1})^{jN_E}$ for $x\in\proj^k\setminus E$ and
\item $\nu(x, [\Psi_{jN_E}])<c_\Psi jN_E(d^{k+1}(\lambda^k\lambda')^{-1})^{jN_E}$ for $x\in\Psi_{jN_E}\setminus E$,
\end{enumerate}
where $c_\Psi$ denotes the number of the irreducible components in the hypersurface $\Psi_1$ of the critical values of $f$.
\end{lem}

\begin{proof}
The first assertion is from Theorem \ref{thm:kappa-} and the second one from Lemma \ref{lem:multiplicitycondition1}. We prove the last assertion. For notational convenience, in the rest of the proof, $n$ denotes $jN_E$.

Let $\Psi'_n$ denote an irreducible hypersurface in $\Psi_n$. Let $\set{\Phi'_{n,l}}_l$ denote the irreducible hypersurfaces in $f^{-n}(\Psi'_n)$. Then, for each $l$, we have that $(f^n)_*(\Phi'_{n,l})=\Psi'_n$ and that $f^n$ induces a ramified covering of $\Phi'_{n,l}$ over $\Psi'_n$ with the covering number $m'_{n,l}$. Then we have $m'_{n,l}[\Psi'_n]=m'_{n,l}[(f^n)_*(\Phi'_{n,l})]=(f^n)_*[\Phi'_{n,l}]$ in the current sense. For this computation, for example, see Example 3.9.2 in \cite{Demailly}. Since $E_\lambda\subseteq E$ and $n$ is an integral multiple of $n_\lambda$, Lemma \ref{lem:multiplicitycondition1} implies that $\mu_{-n}(x)$ for $x\in\proj^k\setminus E$ is uniformly bounded by $(d\lambda^{-1})^n$. By Theorem 3.9.12 in \cite{Demailly}, we have that for $x\in\Psi'_n\setminus E$,
\begin{eqnarray*}
m'_{n,l}\nu(x, [\Psi'_n])&=&\nu(x, (f^n)_*[\Phi'_{n,l}])\leq\sum_{y\in \Phi'_{n,l}\cap f^{-n}(x)} \left(\frac{d}{\lambda}\right)^{nk}\nu(y, [\Phi'_{n,l}])\\
&\leq& m'_{n,l}\left(\frac{d}{\lambda}\right)^{nk}\max_{y\in \Phi'_{n,l}\cap f^{-n}(x)}\nu(y, [\Phi'_{n,l}]).
\end{eqnarray*}

Since $n$ is an integral multiple of $n'_{\lambda'}$, Inequality \ref{ineq:maxmulticrit} induces that for $x\in\Psi'_n\setminus E$,
\begin{equation}\label{ineq:irreduciblecasebound}
\nu(x, [\Psi'_n])< \left(\frac{d}{\lambda}\right)^{nk}\max_{y\in \Phi'_{n,l}\cap f^{-n}(x)}\nu(y, [\Phi'_{n,l}])<\left(\frac{d^{k+1}}{\lambda^k\lambda'}\right)^n.
\end{equation}
Observe that this inequality is independent of the choice of $\Psi'_n$ and $\Phi'_{n,l}$.

We consider the general case. By the chain rule, $\Phi_n=\cup_{i=0}^{n-1}f^{-i}(\Phi_1)$ and therefore, $\Psi_n=\cup_{i=0}^{n-1}f^{i}(\Psi_1)$. Since $f^i$ sends an irreducible hypersurface to an irreducible hypersurface for $i=0, ..., n-1$, the number of irreducible hypersurfaces in $\Psi_n$ is bounded by $c_\Psi n$. Together with Inequality \ref{ineq:irreduciblecasebound}, we obtain the general case, that is, for $x\in\Psi_n\setminus E$,
\begin{displaymath}
\nu(x, [\Psi_n])<c_\Psi n\Big(\frac{d^{k+1}}{\lambda^k\lambda'}\Big)^n.
\end{displaymath}
\end{proof}

\section{Lojasiewicz Type Inequalities}\label{sec:lineq}
In this section, we use Lojasiewicz type inequalities to approximate a certain type of q-psh functions in Lemma \ref{lem:lojasiewicz}. We start by recalling two Lojasiewicz type inequalities. In this section, we denote the $m$-th power of $z$ by $[z]^m$ to distinguish it from $z^m$ with an index $m$.  Let $B_a(r)$ denote a ball centered at $a\in\cplx^k$ and of radius $r>0$ with respect to the Euclidean distance in $\cplx^k$. Let $\pi$ denote the canonical projection from $\cplx^k\times\cplx^k$ onto its first factor just in the next two propositions.

\begin{prop}[See ~\cite{equidist_holo}, ~\cite{cplxdynII}]\label{prop:lojasiewicz}
Let $X$ be an analytic subset of $B_0(1)\times B_0(1)$ of pure dimension $k$ and $s$ a fixed integer. Assume that $\pi:X\to B_0(1)$ defines a ramified covering of degree $m\leq s$ over $B_0(1)$. Then there is a constant $c_1>0$ such that if $x, y$ are two points in $B_0(3/4)$ we can write
\begin{displaymath}
\pi^{-1}(x)\cap X=\set{x^1, ..., x^m}\,\,\,\,\, and\,\,\,\,\,\pi^{-1}(y)\cap X=\set{y^1, ..., y^m}
\end{displaymath}
with $\norm{x^i-y^i}\leq c_1\norm{x-y}^{1/m}$. Moreover, the constant $c_1$ depends on $s$ but not on $X$. Note that the points in the fibers $\pi^{-1}(x)\cap X$ and $\pi^{-1}(y)\cap X$ are repeated according to their multiplicities.
\end{prop}

\begin{prop}[See Proposition 4.2 in ~\cite{equidist_holo}]\label{prop:lojasiewicz1}
Let $X$ be an analytic subset of $B_0(1)\times B_0(1)$ of pure dimension $k$ and let $\delta$ be an integer. Assume $\pi:X\to B_0(1)$ defines a ramified covering of degree $m$ over $B_0(1)$. Let $Z\subset B_0(1)$ be a proper analytic set such that the multiplicity of every point in $\pi^{-1}(x)\cap X$ is at most equal to $\delta<m$ for $x\in B_0(1)\setminus Z$. Then there are constants $c_2>0, N_2\geq 1$ such that for any $0<t<1$ and all $x, y\in B_0(1/2)$ with $\dist_{\mathrm{euc}}(x, Z)\geq t$ and $\dist_{\mathrm{euc}}(y, Z)\geq t$, we can write
\begin{displaymath}
\pi^{-1}(x)\cap X=\set{x^1, ..., x^m}\,\,\,\textrm{ and }\,\,\,\pi^{-1}(y)\cap Y=\set{y^1, ..., y^m}
\end{displaymath}
with $\norm{x^i-y^i}\leq c_2t^{-N_2}\norm{x-y}^{1/\delta}$.
\end{prop}

Let $V$ be an analytic hypersurface in $\proj^k$ of degree $d_V$. We denote by $[V]$ the current of integration over $V$, which is of mass $d_V$. Then, we can find a unique negative q-psh function $\varphi_V$ over $\proj^k$ such that $\sup_{\proj^k}\varphi_V=0$ and $dd^c\varphi_V=[V]-d_V\omega$. Let $0<\delta\leq d_V$ be given. Let $E_V$ be an analytic subset of $V$ such that for all $P\in V\setminus E_V$, $\multi_P(V)<\delta$. We prove the following lemma:
\begin{lem}\label{lem:lojasiewicz}
There are constants $C, A>0$ such that on $\proj^k$,
\begin{equation}\label{ineq:lojasiewicz}
\delta \log\dist(\cdot, V)+C\log\dist(\cdot, E_V)-A\leq \varphi_V \leq \log \dist(\cdot, V)+A.
\end{equation}
\end{lem}

We use the proofs of Proposition 4.2 in ~\cite{equidistributionspeed} and Lemma 2.2.5 in ~\cite{superpotentials}.

\begin{proof}
We first consider the left-side of Inequality \ref{ineq:lojasiewicz}. Since $V$ is compact and $\varphi_V$ is smooth outside $V$, this problem is of local nature. It is enough to prove Inequality \ref{ineq:lojasiewicz} in a neighborhood $U_x$ of $x\in V$.

Consider a point $x\in V$ and denote $m=\multi_x(V)$. Then, we can find small polydiscs $\Delta_x\subset\cplx^k$ of complex dimension $k$ centered at $x$ and $\Delta'_x\subset\Delta_x$ of complex dimension $k-1$, and a projection map $\pi:\Delta_x\cap V\to\Delta'_x$ defining a $m$-sheeted ramified covering over $\Delta'_x$. Fix a local coordinate chart $(z_1, ..., z_k)$ for $\Delta_x$ such that $\pi(z_1, ..., z_k)=(z_1, ..., z_{k-1})$. For convenience, we write $(z_1, ..., z_k)=(z', z_k)$.

For each $z'\in\Delta'_x$, we can find $m$ points $\set{z^1, \cdots, z^m}$ such that $\pi(z^j)=z'$ where $z^j$'s are repeated according to their multiplicities. Let ${z^j}_k$ denote the $k$-th coordinate of $z^j$. Consider a Weierstrass polynomial $H(z)$ defined by
\begin{equation}\label{eqn:H}
H(z):=(z_k-{z^1}_k)\cdots(z_k-{z^m}_k).
\end{equation}
Its zero set is $V$ and it is a holomorphic function on $\Delta_x$. Shrink $\Delta_x$ and $\Delta'_x$ correspondingly with the ramified covering structure preserved, if necessary. Then, since $\varphi_V(z)-\log\abs{H(z)}$ is a smooth function in a neighborhood of $\overline{\Delta_x}$, we only need to show that in a neighborhood $U_x\subseteq \Delta_x$ of $x\in V$,
\begin{equation}\label{ineq:2ndlojasiewicz}
\delta \log{\dist(\cdot, V)}+C_x\log{\dist(\cdot, E_V)}-A_x\leq \log\abs{H(\cdot)}
\end{equation}
for some constants $C_x, A_x>0$.

For this inequality, we claim that there exists a neighborhood $U_x\subseteq \Delta_x$ of $x\in V$ such that for any $z\in U_x$, the ball $B_z$ with its center at $z$ and of radius $\frac{1}{3}{\gamma_3}(\frac{1}{2}\dist_{\mathrm{euc}}(z, E_V))^{N_3\delta}$ with respect to the Euclidean metric intersects at most $\delta$ irreducible components of $V$ for some ${\gamma_3}>0$ and $N_3>0$ both of which are independent of $z$. Our claim works as follows. From Equation \ref{eqn:H}, we see that $\abs{H(z)}$ is bounded below by the product of the distances to each irreducible component of $V$ in $\Delta_x$. If $\delta<m$, then since $B_z$ intersects at most $\delta$ irreducible components of $V$, we have that in $U_x$
\begin{displaymath}
\dist_{\mathrm{euc}}(\cdot, V)^\delta\left[\frac{1}{3}{\gamma_3}\left(\frac{1}{2}\dist_{\mathrm{euc}}(\cdot, E_V)\right)^{N_3\delta}\right]^{m}\leq \abs{H(\cdot)}
\end{displaymath}
and then,
\begin{displaymath}
\delta\log \dist_{\mathrm{euc}}(\cdot, V)+m\log \left(\frac{1}{3}{\gamma_3}\left(\frac{1}{2}\dist_{\mathrm{euc}}(\cdot, E_V)\right)^{N_3\delta}\right)\leq\log\abs{H(\cdot)}.
\end{displaymath}
So, since the Fubini-Study metric and the Euclidean metric are locally equivalent, we take $C_x=m\delta N_3$ and $A_x=-m\log\left(\frac{1}{3}{\gamma_3}\left(\frac{1}{2}\right)^{N_3\delta}\right)+c_{\mathrm{FE}, x}$ for $U_x$, where $c_{\mathrm{FE}, x}$ is a constant due to the difference of the Fubini-Study metric and the Euclidean metric in $\Delta_x$.
If $\delta\geq m$, our claim becomes trivial and $C_x=0, A_x=c_{\mathrm{FE}, x}$ in $\frac{1}{2}\Delta_x$. In particular, if $x\in V\setminus E_V$, then, $C_x=0, A_x=c_{\mathrm{FE}, x}$ in $\frac{1}{2}\Delta_x$. Hence, since $E_V$ is compact, without loss of generality, we may assume that $x\in E_V$ and prove our claim. In the remaining of the proof, the distance is measured with respect to the Euclidean distance. Express $H(z)$ as
\begin{displaymath}
H(z)=[z_k]^m+a_{m-1}(z')[z_k]^{m-1}+\cdots+a_0(z'),
\end{displaymath}
where $a_l(z')$ are holomorphic functions of $z'$ over $\Delta'_x$ for $l=0, ..., m-1$. Define an analytic subset $\Gamma\subset\Delta'_x\times\Delta_x^m$ by $(z', z^1, ..., z^m)\in\Gamma$ if and only $\sum {z^i}_k=-a_{m-1}(z')$, $\sum_{i<j} {z^i}_k{z^j}_k=a_{m-2}(z')$, ..., $\prod {z^i}_k=(-1)^ma_0(z')$. Define $\pi^m_{q+1}:\Delta'_x\times\Delta_x^m\to\Delta'_x\times\Delta_x^{q+1}$ by $\pi^m_{q+1}(z', z^1, ..., z^m)=(z', z^1, ..., z^{q+1})$. Remmert's Proper Mapping Theorem implies that the image $\pi_{q+1}^m(\Gamma)\subset\Delta'_x\times\Delta_x^{q+1}$ of $\Gamma$ under $\pi_{q+1}^m$ is an analytic subset. We denote $\Gamma_{q+1}:=\pi_{q+1}^m(\Gamma)$. Observe that all $z^j$'s for $j=1, ..., q+1$ in $(z', z^1, ..., z^{q+1})\in\Gamma_{q+1}$ have the same first $k-1$ coordinates which equal $z'$.

The following is a modification of Lemma 4.3. in \cite{equidistributionspeed}.
For each $z=(z', z_n)\in\Delta_x\cap V$, define
\begin{displaymath}
h(z)=\min_{\substack{P_z\in\Gamma_{\delta+1}
\\z^1=z}}\sum_{1\leq i, j\leq \delta+1}\norm{z^i-z^j},
\end{displaymath}
where $P_z=(z', z^1, ..., z^{\delta+1})\in\Gamma_{\delta+1}$, and $z^i$ and $z^j$ are the $i$-th and $j$-th coordinates of $P_z$, respectively.

\begin{lem}\label{lem:heightapprox}
There exist $M_3\geq 1$ and ${A_3}>0$ such that for $z\in\frac{1}{2}\Delta_x\cap V$,
\begin{displaymath}
h(z)\geq {A_3}^{-1}\dist_{\mathrm{euc}}(z, E_V)^{M_3}
\end{displaymath}
\end{lem}

\begin{proof}
Let $X_1$ be the set of points $(z, z', z^I)\in\Delta_x\times\Delta'_x\times(\cplx^k)^{(\delta+1)^2}$ for $P_z=(z', z^1, ..., z^{q+1})\in\Gamma_{\delta+1}$ where $z=z^1$ and $z^I$ is defined by $z^i-z^j$ with $1\leq i, j\leq \delta+1$. Since $\Gamma_{\delta+1}$ is analytic, $X_1$ is analytic. The set $X_1$ can be seen as a ramified covering over $\Delta'_x$. Let $X_2:=\Delta_x\times\Delta'_x\times\set{0}\subset \Delta_x\times\Delta'_x\times(\cplx^k)^{(\delta+1)^2}$.

A Lojasiewicz inequality (for example, see p. 14, p. 62 in \cite{idealsofdiffftns}) proves that for sufficiently large ${A_3}, M_3$, for all $P\in X_1\cap (\frac{1}{2}\Delta_x\times\frac{1}{2}\Delta'_x\times(\cplx^k)^{(\delta+1)^2})$,
\begin{displaymath}
\dist_{\mathrm{euc}}(P, X_2)\geq {A_3}^{-1}\dist_{\mathrm{euc}}(P, X_1\cap X_2)^{M_3}.
\end{displaymath}
Thus, for $z\in \frac{1}{2}\Delta_x\cap V$,
\begin{displaymath}
h(z)\geq \dist_{\mathrm{euc}}(P, X_2)\geq {A_3}^{-1}\dist_{\mathrm{euc}}(P, X_1\cap X_2)^{M_3}\geq {A_3}^{-1}\dist_{\mathrm{euc}}(z, E_V)^{M_3}.
\end{displaymath}
For the last inequality, the multiplicity of $V$ is $<\delta+1$ outside $E_V$.
\end{proof}

We continue the proof of our claim. Define $N_3:=M_3m$ with $M_3$ in the previous lemma. Fix a sufficiently small ${\gamma_3}>0$. Consider $z=(z', z_k)\in V$ such that $\norm{z}\leq\frac{1}{4}r$ in the coordinate system of $\Delta_x$ where $r$ is the smallest polyradius of $\Delta_x$. Take $t=\frac{1}{2}\dist_{\mathrm{euc}}(z, E_V)$. Let $B'$ be the ball of center $z'$ and of radius ${\gamma_3} t^{N_3\delta}$ in $\Delta_x'$. For sufficiently small ${\gamma_3}>0$, we have that for all $z\in V$ with $\norm{z}\leq\frac{1}{4}r$ every open neighborhood of the form $B'\times B_{z_k}((8m+2)c_3{\gamma_3}^{1/m}t^{N_3\delta/m})$ sits inside $\Delta_x$, where $c_3$ is the constant from the application of Proposition \ref{prop:lojasiewicz} to this case and $B_{z_k}((8m+2)c_3{\gamma_3}^{1/m}t^{N_3\delta/m})$ denotes the ball in $\cplx$ of center $z_k$ and of radius $(8m+2)c_3{\gamma_3}^{1/m}t^{N_3\delta/m}$. Therefore, we can use the coordinate system of $\Delta_x$ for such open neighborhoods.

We show that the connected component of $V\cap\pi^{-1}(B')$ containing $z$ defines a ramified covering of degree $\leq \delta$ over $B'$ and estimate the size of a polydisc containing the connected component. As in Lemma 4.3 of \cite{equidistributionspeed}, there is an integer $2\leq l\leq 8m$ such that if $\tilde{z}=(z', \tilde{z_k})\in\pi^{-1}(z')\cap V$, we have either $\norm{z-\tilde{z}}\leq(l-2)c_3{\gamma_3}^{1/m}t^{N_3\delta/m}$ or $\norm{z-\tilde{z}}\geq(l+2)c_3{\gamma_3}^{1/m}t^{N_3\delta/m}$.

Let $B''$ denote the ball of center $z_k$ and of radius $lc_3{\gamma_3}^{1/m}t^{N_3\delta/m}$ in $\cplx$ and $\partial B''$ its boundary. Note that $\dist(\pi^{-1}(z')\cap V, B'\times\partial B'')>c_3{\gamma_3}^{1/m}t^{N_3\delta/m}$. Proposition ~\ref{prop:lojasiewicz} gives us $\dist_{\mathrm{euc}}(y, \pi^{-1}(z')\cap V)\leq c_3\norm{y'-z'}^{1/m}$, where $y'\in B'$ and $y\in\pi^{-1}(y')\cap V$. So, $V\cap(B'\times\partial B'')=\emptyset$ and therefore, $\pi$ is proper on $V\cap(B'\times B'')$ over $B'$.

We want to show that the degree of the covering is $\leq \delta$. If not, $h(z)$ is of order ${\gamma_3}^{1/m}t^{N_3\delta/m}\ll {A_3}^{-1}t^{N_3/m}={A_3}^{-1}t^{M_3}$ because ${\gamma_3}$ is small. Note that $z$ obviously sits inside $\frac{1}{2}\Delta_x$ from our choice of $z$ such that $\norm{z}\leq\frac{1}{4}r$ and therefore, Lemma ~\ref{lem:heightapprox} is valid. This is a contradiction since $h(z)\geq {A_3}^{-1}\dist_{\mathrm{euc}}(z, E_V)^{M_3}\geq {A_3}^{-1}t^{M_3}$ from Lemma ~\ref{lem:heightapprox}.

So, the box $B'\times B''$ contains at most $\delta$ locally irreducible components. The ball of radius ${\gamma_3} t^{N_3\delta}=\min\set{c_3{\gamma_3}^{1/m}t^{N_3\delta/m}, {\gamma_3} t^{N_3\delta}}$ 
and of center $z\in V$ such that $\norm{z}\leq\frac{1}{4}r$ can intersect at most $\delta$ irreducible components and the distances to the other irreducible components from $z$ are $>\gamma_3 t^{N_3\delta}$.

Take the neighborhood $U_x$ of $x\in E_V$ 
to be
the ball $\set{\norm{z}<\frac{1}{4}r}\subseteq \Delta_x$ in the coordinate system of $\Delta_x$. This is the desired open neighborhood of $x$. Indeed, for an arbitrary point $z\in U_x$, the ball $B_z$ centered at $z$ and of radius $\frac{1}{3}{\gamma_3}(\frac{1}{2}\dist_{\mathrm{euc}}(z, E_V))^{N_3\delta}$ can contain at most $\delta$ irreducible components of $V$. Hence, our claim is proved.

The right-side of Inequality ~\ref{ineq:lojasiewicz} is from Lemma 2.2.5 in ~\cite{superpotentials}. It is simply from the compactness of $V$.
\end{proof}

\section{Currents}\label{sec:currents}
In this paper, we assume some familiarity of the reader to currents. For details, see, for example, \cite{Demailly}, Appendix of \cite{dynamicsinscv}, \cite{federer}, \cite{lelong}, \cite{ueda} and \cite{sibony}. For the convenience of the reader, we summarize some concepts, properties and remarks.

For a positive $(p,p)$-current $S$ on $\proj^k$, the mass of $S$ is equivalent to 
\begin{displaymath}
\norm{S}:=\agl{S, \omega^{k-p}}=\int_{\proj^k}S\wedge\omega^{k-p}.
\end{displaymath}
In the same way, for negative $(p,p)$-currents, the mass of $S'$ is equivalent to $\norm{S'}:=\agl{-S',\omega^{k-p}}$. 
In the rest of this paper, we denote by $\cali{C}_p$ the set of positive closed $(p,p)$-currents of mass $1$.

We recall some properties of $\cali{C}_p$. The space $\cali{C}_p$ is a Polish space. It is compact in the current sense. Moreover, the weak topology on $\cali{C}_p$ is metrizable (see \cite{superpotentials}). Indeed, for $\alpha\geq 0$, let $[\alpha]$ denote the largest integer $\leq \alpha$. Let $\cali{C}_q^\alpha$ be the space of $(q, q)$-forms whose coefficients admit derivatives of all orders $\leq [\alpha]$ and these derivatives are $(\alpha-[\alpha])$-H\"{o}lder continuous. We use here the sum of $\cali{C}^\alpha$-norms of the coefficients for a fixed atlas. If $R$ and $R'$ are currents in $\cali{C}_p$, we define
\begin{displaymath}
\dist_\alpha(R, R'):=\sup_{\norm{\Phi}_{\cali{C}^\alpha}\leq 1}\abs{\agl{R-R', \Phi}},
\end{displaymath}
where $\Phi$ is a smooth $(k-p, k-p)$-form on $\proj^k$.
\begin{prop}[See Proposition 2.1.4 in \cite{superpotentials}]
Then, the topology induced by the distance $\dist_\alpha$ with any $\alpha>0$ is equivalent to the weak topology on $\cali{C}_p$.
\end{prop}

Another property of $\cali{C}_p$ is that the smooth forms are dense in $\cali{C}_p$. 
We regularize $R\in\cali{C}_p$ using the automorphisms of $\proj^k$ in the following way (for example, see \cite{superpotentials}). Consider an endomorphism $h_\theta$ of $\set{\norm{\zeta}_A<1}$ defined by $h_\theta(\zeta):=\theta \zeta$ for $\theta\in\cplx$ with $\abs{\theta}<1$, where $\set{\norm{\zeta}_A<1}$ is the coordinate chart of $\Aut(\proj^k)$ introduced in Introduction. Define $\rho_\theta:=(h_\theta)_*(\rho)$. 
\begin{defn}
For any positive or negative $(p, p)$-current $R$ on $\proj^k$, we define the \emph{$\theta$-regularization $R_\theta$} of $R$ by
\begin{displaymath}
R_\theta:=\int_{\Aut(\proj^k)}(\tau_\zeta)_*Rd\rho_\theta(\zeta)=\int_{\Aut(\proj^k)}(\tau_{\theta \zeta})_*Rd\rho(\zeta)=\int_{\Aut(\proj^k)}(\tau_{\theta \zeta})^*Rd\rho(\zeta),
\end{displaymath}
where $\tau_\zeta$ is the automorphism in $\Aut(\proj^k)$ whose coordinate is $\zeta$.
\end{defn}
If $R$ is positive and closed, then so is $R_\theta$. If $\abs{\theta}=\abs{\theta'}$, then $R_\theta=R_{\theta'}$. The mass of $R_\theta$ does not depend on $\theta$ (see Lemma 2.4.1 in \cite{superpotentials}). The $\theta$-regularization has the following estimate property:

\begin{prop}[See Proposition 2.1.6 in ~\cite{superpotentials}]\label{prop:216}
If $\theta\neq 0$, then $R_\theta$ is a smooth form which depends continuously on $R$. Moreover, for every $\alpha\geq 0$, there is a constant $c_\alpha$ independent of $R$ such that
\begin{displaymath}
\norm{R_\theta}_{\cali{C}^\alpha}\leq c_\alpha\norm{R}\abs{\theta}^{-2k^2-4k-\alpha}.
\end{displaymath}
If $K$ is a compact subset in $\set{\theta\in\cplx:\abs{\theta}<1}\setminus\set{0}$, then there is a constant $c_{\alpha, K}>0$ such that for $\theta, \theta'\in K$,
\begin{displaymath}
\norm{R_\theta-R_{\theta'}}_{\cali{C}^\alpha}\leq c_{\alpha, K}\norm{R}\abs{\theta-\theta'}.
\end{displaymath}
\end{prop}

Also, when $R$ is of class $\cali{C}^\alpha$ for $\alpha\geq 0$, we have
\begin{lem}[See Lemma 2.1.8 in ~\cite{superpotentials}]\label{lem:218}
Let $K$ be a compact subset of $\Aut(\proj^k)$. Let $W$ and $W_0$ be open sets in $\proj^k$ such that $\overline{W_0}\subset\tau(W)$ for every $\tau\in K$. If $R$ is of class $\cali{C}^\alpha$ on $W$ with $\alpha\geq 0$, then $\tau_*(R)$ is of class $\cali{C}^\alpha$ on $W_0$. Moreover, there is a constant $\tilde{c}>0$ such that for all $\tau, \tau'\in K$,
\begin{displaymath}
\norm{\tau_*(R)}_{\cali{C}^\alpha(W_0)}\leq \tilde{c}\norm{R}_{\cali{C}^\alpha(W)}
\end{displaymath}
and
\begin{displaymath}
\norm{\tau_*(R)-\tau_*'(R)}_{\cali{C}^\alpha(W_0)}\leq \tilde{c}\norm{R}_{\cali{C}^\alpha(W)}\dist_\Aut(\tau, \tau')^{\min\set{\alpha, 1}},
\end{displaymath}
where the distance $\dist_\Aut(\cdot, \cdot)$ is with respect to a fixed smooth metric on $\Aut(\proj^k)$.
\end{lem}


We briefly introduce the Green $(p,p)$-current $T^p$ on $\proj^k$ associated with $f$. It can be defined by $T^p:=\lim_{n\to\infty}d^{-pn}(f^n)^*(\omega^p)$.
The pull-back $f^*$ on $\cali{C}_p$ is well-defined and $\norm{f^*(S)}=d^p$ for $S\in\cali{C}_p$. For details, see \cite{pullbackbyholo} and \cite{superpotentials}. We say that $S\in\cali{C}_p$ is invariant under $f$ if $\frac{1}{d^p}f^*(S)=S$. The Green $(p,p)$-current $T^p$ is a positive closed $(p,p)$-current of mass $1$ with some special properties: it is invariant under $f$ and it is an extreme point in the convex subset $\set{S\in\cali{C}_p:\frac{1}{d^p}f^*(S)=S}$ of $\cali{C}_p$. It is most diffuse in $\set{S\in\cali{C}_p:\frac{1}{d^p}f^*(S)=S}$. The super-potential of $T^p$ (cf. Section \ref{sec:superpotentials}) is H\"{o}lder continuous. It is uniquely determined for our $f$. For details, see \cite{superpotentials}. 
\medskip

We close this section by recalling a class of functions which generalizes the class of q-psh functions. 
\begin{defn}[See ~\cite{DSHreference} and p.12 in ~\cite{superpotentials}]
An integrable function $\varphi$ on $\proj^k$ is said to be DSH if it is equal outside a pluri-polar set to a difference of two q-psh functions. 
\end{defn}
We identify two DSH functions if they coincide outside a pluripolar set. The set of DSH functions is a vector space over $\Real$. The space of DSH functions is endowed with the following norm:
\begin{displaymath}
\norm{\varphi}_\dsh:=\norm{\varphi}_{\cali{L}^1}+\inf\set{\norm{T^+}: dd^c\varphi=T^+-T^-, T^\pm \textrm{ are positive and closed. }}.
\end{displaymath}
The currents $T^+$ and $T^-$ are cohomologous and have the same mass. 

\section{Super-potentials}\label{sec:superpotentials}
In this section, we briefly introduce the main tool, super-potentials, developed by Dinh-Sibony. For details, see \cite{superpotentials}.

\begin{defn}[See Section 3 in \cite{superpotentials}]
Let $S$ be a smooth form in $\cali{C}_p$ and $m$ a fixed real number. Then the super-potential $\cali{U}_S$ of $S$ of mean $m$ is a function on $\cali{C}_{k-p+1}$ defined by
\begin{displaymath}
\cali{U}_S(R):=\agl{S, U_R} \textrm{ for }R\in\cali{C}_{k-p+1},
\end{displaymath}
where $U_R$ is a quasi-potential of mean $m$ of $R$.

In general, for $S\in\cali{C}_p$, $\cali{U}_S(R)$ is defined by
\begin{displaymath}
\cali{U}_S(R):=\lim_{\theta\to 0}\cali{U}_{S_\theta}(R)
\end{displaymath}
where the subscript $\theta$ means the $\theta$-regularization of the current.
\end{defn}

\begin{prop}[See Section 3 in \cite{superpotentials}] Assume that $R$ is smooth. Then, $\cali{U}_S(R)=\agl{U_S, R}$, where the super-potential $\cali{U}_S(\cdot)$ of $S\in\cali{C}_p$ is of mean $m$ and $U_S$ is a quasi-potential of $S$ of mean $m$. In particular, for $S\in\cali{C}_p$ and $R\in\cali{C}_{k-p+1}$, we have $\cali{U}_S(R)=\lim_{\theta\to 0}\cali{U}_{S}(R_\theta)=\lim_{\theta\to 0}\cali{U}_{S_\theta}(R)$, where the subscript $\theta$ means the $\theta$-regularization of the current and all the super-potentials are assumed to be of the same mean $m$.
\end{prop}

Observe that the definition of the super-potential does not depend on the choice of a quasi-potential $U_R$ nor $U_S$ as long as their means equal $m$. By choosing a canonical-type quasi-potential, we can obtain good estimates for super-potentials. Such a quasi-potential is called the Green quasi-potential and defined using Proposition \ref{prop:kernel} below.

\begin{prop}[See Proposition 2.3.2 in ~\cite{superpotentials}]\label{prop:kernel}
Consider $X:=\proj^k\times\proj^k$ and $D$ the diagonal of $X$. Let $\Omega(z, \xi):=\sum_{j=0}^k\omega(z)^j\wedge\omega(\xi)^{k-j}$, where $(z, \xi)$ denotes the homogeneous coordinates of $\proj^k\times\proj^k$ with $z=[z_0: ...: z_k]$ and $\xi=[\xi_0: ...: \xi_k]$. Then, there is a negative $(k-1, k-1)$-form $K$ on $X$ smooth outside $D$ such that $dd^cK=[D]-\Omega$ which satisfies the following inequalities near $D$:
\begin{displaymath}
\norm{K(\cdot)}_\infty\lesssim -\dist(\cdot, D)^{2-2k}\log \dist(\cdot, D)\,\,\,\,\,\textrm{ and }\,\,\,\,\,\norm{\nabla K(\cdot)}_\infty\lesssim \dist(\cdot, D)^{1-2k}.
\end{displaymath}
Moreover, there is a negative DSH function $\eta$ and a positive closed $(k-1, k-1)$-form $\Theta$ smooth outside $D$ such that $K\geq \eta\Theta$, $\norm{\Theta(\cdot)}_\infty\lesssim \dist(\cdot, D)^{2-2k}$ and $\eta-\log\dist(\cdot, D)$ is bounded near $D$.
\end{prop}
In the above, $\norm{\nabla K(\cdot)}_\infty$ is the sum $\sum_j\abs{\nabla K_j}$, where the $K_j$'s are the coefficients of $K$ for a fixed atlas of $X$. On $\proj^k\times\proj^k$, the distance, which is also denoted by $\dist$, is measured with respect to the product metric of the standard Fubini-Study metric on each $\proj^k$.

\begin{lem}[See Lemma 2.3.3 in ~\cite{superpotentials}]\label{lem:233}
There is a negative DSH function $\eta$ on $X$ smooth outside $D$ such that $\eta-\log \dist(\cdot, D)$ is bounded.
\end{lem}

\begin{defn}
The Green quasi-potential $U$ of $R\in\cali{C}_p$ is defined by
\begin{displaymath}
U(z):=\int_{\xi\neq z}R(\xi)\wedge K(z, \xi).
\end{displaymath}
\end{defn}

\begin{rem}\label{rem:Greenqpotential}
Note that $U$ depends on the choice of $K$.
 The mean $m$ of $U$ (or equivalently, the mass of $U$ since $U$ is negative) is bounded uniformly with respect to $R$. Note also that $U-m\omega^{p-1}$ is a quasi-potential of mean $0$ of $R$.
\end{rem}

\begin{thm}[See Theorem 2.3.1]\label{thm:Greenqpotential}
Let $R$ be a current in $\cali{C}_p$. Then, its Green quasi-potential $U$ is negative, depends linearly on $R$ and satisfies that for every $r$ and $s$ with $1\leq r<\frac{k}{k-1}$ and $1\leq s<\frac{2k}{2k-1}$, one has
\begin{displaymath}
\norm{U}_{\cali{L}^r}\leq c_r\,\,\, \textrm{ and }\,\,\,\norm{dU}_{\cali{L}^s}\leq c_s
\end{displaymath}
for some positive constants $c_r$ and $c_s$ independent of $R$. Moreover, $U$ depends continuously on $R$ with respect to $\cali{L}^r$-topology on $U$ and the weak topology on $R$.
\end{thm}

The following shows that super-potentials determine currents.
\begin{prop}[See Proposition 3.1.9 in ~\cite{superpotentials}]\label{prop:319}
Let $I$ be a compact subset in $\proj^k$ with $(2k-2p)$-dimensional Hausdorff measure 0. Let $S$ and $S'$ be currents in $\cali{C}_p$, with super-potentials $\cali{U}_S$ and $\cali{U}_{S'}$. If $\cali{U}_S=\cali{U}_{S'}$ on smooth forms in $\cali{C}_{k-p+1}$ with compact support in $\proj^k\setminus I$, then $S=S'$.
\end{prop}

The following is about the compactness property of super-potentials.
\begin{prop}[See Proposition 3.2.6 in ~\cite{superpotentials}]\label{prop:326}
Let $\cali{U}_{S_n}$ be a super-potential of a current $S_n$ in $\cali{C}_p$. Assume that $\set{\cali{U}_{s_n}}_{n\geq 0}$ is bounded from above and does not converge uniformly to $-\infty$. Then there is an increasing sequence $\set{n_j}_{j\geq 0}$ of integers such that $S_{n_j}$ converge to a current $S$ and $\cali{U}_{S_{n_j}}$ converge on smooth forms in $\cali{C}_{k-p+1}$ to a super-potential $\cali{U}_S$ of $S$. Moreover,
\begin{displaymath}
\limsup_{j\to\infty}\cali{U}_{s_{n_j}}\leq\cali{U}_S.
\end{displaymath}
\end{prop}

The following theorem is about the regularity of the Green super-potentials of order $p$, that is, the super-potentials of the Green $(p,p)$-current.
\begin{thm}[See Theorem 5.4.1 in ~\cite{superpotentials}]\label{thm:541}
Let $f:\proj^k\to\proj^k$ be a holomorphic map of algebraic degree $\geq 2$. Then, the Green super-potentials of order $p$ of $f$ are H\"{o}lder continuous on $\cali{C}_{k-p+1}$ with respect to the distance $\dist_1$.
\end{thm}

\section{The Proof of the Main Theorem}\label{sec:settings}
In this section, we restate Theorem \ref{thm:mainthm} in terms of super-potentials and prove the new statement (see Proposition \ref{prop:mainthm2}). We follow the notation used in Section 5 of ~\cite{superpotentials} for convenience.
\medskip

Let $S\in\cali{C}_p$ and $T^p$ the Green $(p,p)$-current associated with $f$ over $\proj^k$.
\begin{defn}[See Section 5 in \cite{superpotentials}]
The \emph{dynamical super-potential} of $S$, denoted by $\cali{V}_S$, is defined by
\begin{displaymath}
\cali{V}_S:=\cali{U}_S-\cali{U}_{T^p}-c_S,
\end{displaymath}
where $\cali{U}_S$, $\cali{U}_{T^p}$ are the super-potentials of mean $0$ of $S$, $T^p$, respectively, and $c_S:=\cali{U}_S(T^{k-p+1})-\cali{U}_{T^p}(T^{k-p+1})$.
\end{defn}
\begin{defn}[See Section 5 in \cite{superpotentials}]
The \emph{dynamical Green quasi-potential} of $S$ is defined by 
\begin{displaymath}
V_S:=U_S-U_{T^p}-(m_S-m_{T^p}+c_S)\omega^{p-1},
\end{displaymath}
where $U_S$, $U_{T^p}$ are the Green quasi-potentials of $S$, $T^p$, respectively and $m_S$, $m_{T^p}$ are their corresponding means.
\end{defn}

For notational convenience, we use the following notation in the rest of the paper:
\begin{displaymath}
L:=\frac{1}{d^p}f^*\,\,\,\,\, \textrm{and}\,\,\,\,\, \Lambda:=\frac{1}{d^{p-1}}f_*.
\end{displaymath}
For the definitions of these operators on $\cali{C}_p$, see \cite{superpotentials} and also see \cite{pullbackbyholo}.

The dynamical super-potentials have the following properties.

\begin{lem}[See Lemma 5.4.6 in ~\cite{superpotentials}]\label{lem:546}
We have $\cali{V}_S(T^{k-p+1})=0$, $\cali{V}_S(R)=\agl{V_S, R}$ for smooth $R\in\cali{C}_{k-p+1}$, and $\cali{V}_{L(S)}=d^{-1}\cali{V}_S\circ\Lambda$ on $\cali{C}_{k-p+1}$. Moreover, $\cali{U}_S-\cali{V}_S$ is bounded by a constant independent of $S$.
\end{lem}

\begin{lem}[See Lemma 5.4.9 in ~\cite{superpotentials}]\label{lem:549}
If $R\in\cali{C}_{k-p+1}$ is smooth, then $\cali{V}_S(\Lambda(R))=\agl{V_S, \Lambda(R)}_{\proj^k\setminus \Psi_1}$.
\end{lem}

By Proposition ~\ref{prop:319} and Proposition ~\ref{prop:326}, Proposition ~\ref{prop:mainthm2} implies ~Theorem ~\ref{thm:mainthm}.
\begin{prop}\label{prop:mainthm2}
Let $f:\proj^k\to\proj^k$ be a holomorphic endomorphism of degree $d\geq 2$. Then, there is a proper invariant analytic subset $E$ for $f$ such that for any current $S\in\cali{C}_p$ which is smooth on $E$ and for any smooth form $R\in\cali{C}_{k-p+1}$,we have $
\cali{V}_{L^n(S)}(R)\to 0$ as $n\to\infty$. The convergence is exponentially fast.
\end{prop}

\begin{proof}[The proof of Proposition ~\ref{prop:mainthm2}]
It suffices to prove the statement for some iterate of $f$. We use Section ~\ref{sec:analyticcocycle} to find a good iterate of $f$. Let $N_f\in\mathbb{N}$ be such that $(40k^2c_\Psi N_f)^{8k}<d^{N_f}$, where $c_\Psi$ is as in Section ~\ref{sec:analyticcocycle}. Then, we choose $\lambda>0$ such that
\begin{displaymath}
(40k^2c_\Psi N_f)^\frac{1}{(k+1)N_f}\cdot d^{\frac{8k(k+1)-1}{8k(k+1)}}<\lambda<d.
\end{displaymath}
Define $E:=E_\lambda\cup E_\lambda'$ where $E_\lambda:=\set{\mu_-\geq\frac{d}{\lambda}}$ and $E_\lambda':=\set{\mu'_-\geq\frac{d}{\lambda}}$ as in Section ~\ref{sec:analyticcocycle}. This $E$ is the desired invariant analytic subset. By Lemma \ref{lem:multiplicityconditions}, there exists $N_E\in\mathbb{N}$ such that for any $j\in\mathbb{N}$,
\begin{enumerate}
\item the function $\mu_{-jN_E}$ is $<\left(\frac{d}{\lambda}\right)^{jN_E}$ on $\proj^k\setminus E$ and
\item $\nu(x, [\Psi_{jN_E}])<c_\Psi jN_E\left(\frac{d}{\lambda}\right)^{jN_E(k+1)}$ for $x\in\Psi_{jN_E}\setminus E$.
\end{enumerate}
Then, consider $f^{N_fN_E}$ with $\delta:=c_\Psi N_fN_E\left(\frac{d}{\lambda}\right)^{N_fN_E(k+1)}$. From the choice of $\lambda$,
\begin{eqnarray*}
(20k^2(\delta+1/2))^{8k}<(40k^2\delta)^{8k}
\leq\Big[\Big(40k^2c_\Psi N_f\Big(\frac{d}{\lambda}\Big)^{N_f(k+1)}\Big)^{8k}\Big]^{N_E}\leq d^{N_fN_E}.
\end{eqnarray*}
Hence, since $E$ is invariant under $f$, by replacing $f$ by $f^{N_fN_E}$ and $d$ by $d^{N_fN_E}$ and defining $\delta:=c_\Psi N_fN_E\left(\frac{d}{\lambda}\right)^{N_fN_E(k+1)}$, we may assume that $f$ satisfies:
\begin{enumerate}
\item the function $\mu_{-1}<\delta$ on $\proj^k\setminus E$ and $\nu(x, [\Psi_1])<\delta$ for $x\in\Psi_1\setminus E$,
\item the analytic subset $E$ is invariant under $f$, and
\item $(20k^2(\delta+1/2))^{8k}<(40k^2\delta)^{8k}<d$.
\end{enumerate}

For the proof, we modify the proof of Proposition 5.4.10 of \cite{superpotentials}. First, we divide $\proj^k$ into three regions for computational purposes.
For the rest of the paper, let $V$ denote the hypersurface $\Psi_1$ of the critical values of $f$.

Since $E$ is invariant under $f$, $L^j(S)$ is smooth over $E$ for $j=0, 1,...$. Denote by $O$ and ${O_j}$ the connected open neighborhoods of $E$ where $S$ and $L^j(S)$ are smooth, respectively. Consider the following Corollary ~\ref{cor:lojasiewicz} and recall Lemma ~\ref{lem:3.1}.

\begin{cor}[See Proposition \ref{prop:lojasiewicz1} in Section \ref{sec:lineq} or Corollary 4.4 in \cite{equidistributionspeed}]\label{cor:lojasiewicz}
There are an integer $N_4$ and a constant $c_4\geq 1$ such that if $0<t<1$ is a constant and if $x, y$ are two points in $\proj^k$ with $\dist(x, E)>t$ and $\dist(y, E)>t$, then we can write
\begin{displaymath}
f^{-1}(x)=\set{x^1, ..., x^{d^k}}\,\,\,\textrm{ and }\,\,\,f^{-1}(y)=\set{y^1, ..., y^{d^k}}
\end{displaymath}
with $\dist(x^i, y^i)\leq c_4t^{-N_4}\dist(x, y)^{1/\delta}$.
\end{cor}

\begin{lem}[See Lemma 3.1 in ~\cite{equidistributionspeed}]\label{lem:3.1}
There is a constant $A_1\geq 1$ such that for every subsets $X$ and $Y$ of $\proj^k$, we have that for $j\geq 0$,
\begin{displaymath}
\dist(f^{-j}(X), f^{-j}(Y))\geq {A_1}^{-j}\dist(X, Y).
\end{displaymath}
In particular, if $f(Y)\subseteq Y$, we have that for $j\geq 0$,
\begin{displaymath}
\dist(f^{-j}(X), Y)\geq{A_1}^{-j}\dist(X, Y).
\end{displaymath}
\end{lem}

We take $X=O^c$ and $Y=E$. Then, we have
\begin{displaymath}
\dist({O_j}^c, E)\geq \dist(f^{-j}(O^c), E)\geq {A_1}^{-j} \dist(O^c, E).
\end{displaymath}
Since $E$ is compact, there exists $r>0$ with $\dist(O^c, E)>r$. Then, ${O_j}^c\cap E_{r{A_1}^{-j}}=\emptyset$. Let $\eps>0$ be sufficiently small so that $\eps\ll 1$ and $\eps<\min\set{r, {A_1}^{-1}}^2$. We define three sequences $\set{s_{n,i}}, \set{\eps_{n,i}}$ and $\set{t_{n,i}}$ of positive real numbers:
\begin{itemize}
\item $s_{n,i}=\eps^{ni}$,

\item $\eps_{n,i}=\eps^{nC(2+N_4)(40k^2\delta)^{6ki}}$, and

\item $t_{n,i}={\eps_{n,i}}^{(10k)^{-1}}$,
\end{itemize}
for $0\leq i\leq n$, where $N_4$ is the constant in Corollary ~\ref{cor:lojasiewicz} and $C>1$ is the constant in Lemma ~\ref{lem:lojasiewicz}. Let $W_{n,i}'':=E_{s_{n,i}}$, $W_{n,i}:=V_{t_{n,i}}\setminus W_{n,i}''$, and $W_{n,i}':=\proj^k\setminus(W_{n,i}\cup W_{n,i}'')$. From our definitions, $O_{n-i}\cap\frac{1}{2\sqrt{\eps}}W_{n,i}''=\emptyset$, where $\frac{1}{2\sqrt{\eps}}W_{n,i}''$ means the $\frac{1}{2\sqrt{\eps}}s_{n,i}$-neighborhood of $E$. 
Roughly speaking, each part has the following properties:
\begin{itemize}
\item on $W_{n,i}$, $f$ and $V$ have low multiplicities but $L^{n-i}(S)$ can be singular,
\item on $W_{n,i}'$, $f$ and $V$ have multiplicities $=1$ and $L^{n-i}(S)$ is smooth, and
\item on $W_{n,i}''$, $f$ or $V$ can have high multiplicities but $L^{n-i}(S)$ is smooth.
\end{itemize}
Also, observe that $s_{n,i}\gg t_{n,i}$. 
\medskip

Let $U_j$ and $V_j$ denote the Green quasi-potential and the dynamical Green quasi-potential of $L^j(S)$, respectively for $j=0, 1, 2, \cdots$. Let $R_{n,i}$ be the sequence of positive closed $(k-p+1, k-p+1)$-currents of mass $1$ defined by 
\begin{itemize}
\item $R_{n,0}:=R$ and
\item $R_{n,i}$ is the $\eps_{n,i}$-regularization of $\Lambda(R_{n,i-1})$.
\end{itemize}
Then, by Lemma ~\ref{lem:546} and Lemma ~\ref{lem:549}, we have
\begin{eqnarray}
\nonumber\cali{V}_{L^n(S)}(R)&=&d^{-n}\cali{V}_S(\Lambda^n(R))\\
\nonumber&=&d^{-1}\cali{V}_{L^{n-1}(S)}(\Lambda(R_{n,0}))\\
\nonumber&=&d^{-1}\agl{V_{n-1}, \Lambda(R_{n,0})-R_{n,1}}_{\proj^k\setminus
V}+d^{-1}\agl{V_{n-1}, R_{n,1}}\\
\nonumber&=&d^{-1}\agl{V_{n-1}, \Lambda(R_{n,0})-R_{n,1}}_{\proj^k
\setminus V}+d^{-1}\cali{V}_{L^{n-1}(S)}(R_{n,1})\\
\nonumber&&\cdots\\
\label{eqn:expansion}&=&d^{-1}\agl{V_{n-1},
\Lambda(R_{n,0})-R_{n,1}}_{\proj^k\setminus
V}+ \cdots \\
\nonumber&&+d^{-i}\agl{V_{n-i},
\Lambda(R_{n,i-1})-R_{n,i}}_{\proj^k\setminus V}+\cdots\\
\nonumber&&+d^{-n}\agl{V_0, \Lambda(R_{n,n-1})-R_{n,n}}_{\proj^k\setminus
V}+ d^{-n}\cali{V}_S(R_{n,n}).
\end{eqnarray}

Since $\cali{V}_S$ is bounded above by a constant independent of $S$ (by Theorem \ref{thm:541} and Remark \ref{rem:Greenqpotential}), we only need to bound $\cali{V}_{L^n(S)}(R)$ from below.
\medskip

The current $V_{n-i}$ is neither positive nor negative in general. 
From Theorem ~\ref{thm:541}, we know that $T^{k-p+1}$ has bounded super-potentials. 
Hence, there exists a universal $c>0$ such that $V_{n-i}':=V_{n-i}+U_{T^p}-c\omega^{p-1}$ is negative in the current sense. In the rest of this section, the inequalities $\lesssim$ and $\gtrsim$ are up to a constant multiple independent of $\eps, i, n$ and $R$.
\medskip

We claim that for sufficiently small $\eps>0$,
\begin{enumerate}
\item $\agl{V_{n-i}', \Lambda(R_{n,i-1})-R_{n,i}}_{W_{n,i}\setminus V}\gtrsim-\eps^{ni}$,

\item $\agl{V_{n-i}', \Lambda(R_{n,i-1})-R_{n,i}}_{W_{n,i}'}\gtrsim-\eps^{ni}$, 

\item $\agl{V_{n-i}', \Lambda(R_{n,i-1})-R_{n,i}}_{W_{n,i}''\setminus V}\gtrsim-(1+\norm{S}_{\infty, E_{\eps}})\eps^{ni}$, and 

\item $d^{-n}\V_S(R_{n,n})\gtrsim nd^{-\frac{n}{4}}\log\eps$,
\end{enumerate}
where $\norm{S}_{\infty, E_{\eps}}$ is the sup-norm of $S$ restricted to $E_{\eps}$.

Suppose that our claims are true. Note that for a fixed $S\in\cali{C}_p$ smooth on $E$, $\norm{S}_{\infty, E_{\eps}}$ is bounded for all sufficiently small $\eps>0$. It is not difficult to see that $\dist_1(\Lambda(R_{n, i-1}), R_{n, i})\lesssim\eps_{n, i}$. Due to Equation \ref{eqn:expansion}, Theorem ~\ref{thm:541}
and the constant $c$ in $V_{n-i}'$ being universal,
we have that for all sufficiently small $\eps>0$,
\begin{eqnarray*}
\cali{V}_{L^n(S)}(R)&=&d^{-n}\cali{V}_S(\Lambda^n(R))\\
&\gtrsim& -d^{-1}\eps^n-d^{-2}\eps^{2n}\cdots-d^{-n}\eps^{n^2}+nd^{-\frac{n}{4}}\log\eps\gtrsim-\eps^n+nd^{-\frac{n}{4}}\log\eps.
\end{eqnarray*}
Take $\eps=d^{-n}$ and let $n\to\infty$. Our claims imply the proof.
\end{proof}

Hence, it remains to prove our claims. The computations will be completed in Section \ref{sec:1}, Section \ref{sec:2}, Section \ref{sec:3} and Section \ref{sec:4}.

\section{The Proof of $\agl{V_{n-i}', \Lambda(R_{n,i-1})-R_{n,i}}_{W_{n,i}\setminus V}\gtrsim-\eps^{ni}$}\label{sec:1}
The main estimate Proposition \ref{prop:estimate_off_E} in this section is similar to Proposition 2.3.6 in \cite{superpotentials} but has an extra difficulty; we have to handle neighborhoods of $E$. For this purpose, we use cut-off functions for the neighborhoods of $E$, whose DSH-norm shows very tame behaviors. We also use the H\"{o}lder continuity of the quasi-potentials of $f_*(\omega)$ together with the current inequality $f_*(\omega^p)\leq (f_*\omega)^p$. Here, H\"older continuity must be handled with some care about $E$.
\medskip

We start by constructing cut-off functions from the following two lemmas. By applying the $c_{\mathrm{small}}\theta$-regularization of currents to the characteristic function of the neighborhood $K_\theta$ of a compact set $K\subset \proj^k$ for a sufficiently small constant $c_{\mathrm{small}}>0$ independent of $K$ and $\theta$, we obtain the following:
\begin{lem}\label{lem:smoothing}
Assume that $\theta$ is sufficiently small. Let $K\subset\proj^k$ be compact. Then, there exists a smooth cut-off function $\chi^{K, \theta}:\proj^k\to[0, 1]$ such that $\chi^{K, \theta}\equiv 1$ over $K$ and $\supp(\chi^{K, \theta})\subset K_\theta$. Moreover, $\norm{\chi^{K, \theta}}_{\cali{C}^\alpha}\leq c_{\chi, \alpha}\abs{\theta}^{-\alpha}$, for some $c_{\chi, \alpha}>0$ independent of $K$ and $\theta$.
\end{lem}

\begin{lem}[See Lemma 2.2.6 in ~\cite{superpotentials}]\label{lem:cvxincftn}
Let $\chi:\Real\cup\set{-\infty}\to\Real$ be a convex increasing function such that $\chi'$ is bounded. Then, for every DSH function $\varphi$, $\chi(\varphi)$ is DSH and
\begin{displaymath}
\norm{\chi(\varphi)}_\dsh\lesssim 1+\norm{\varphi}_\dsh.
\end{displaymath}
In particular, $\inf\norm{T^-_\chi}$ is bounded by $\norm{\chi'}_\infty\inf\norm{T^-}$, where $T^\pm_\chi$ and $T^\pm$ are positive closed currents such that $dd^c\chi(\varphi)=T^+_\chi-T^-_\chi$ and $dd^c\varphi=T^+-T^-$ and $\inf$ is taken over such $T^-_\chi$'s and $T^-$'s, respectively.
\end{lem}

We construct the desired cut-off functions.
\begin{lem}\label{lem:cut-off}
Let $s, t$ be two positive real numbers such that $\frac{1}{2}(\frac{2}{3})^ks_{n,i}\leq s\leq 2s_{n,i}$ and $\frac{1}{2}t_{n,i}\leq t\leq 2{t_{n,i}}^{(\delta+1/2)^{-k}}$ for some $s_{n,i}, t_{n,i}$ with sufficiently small $\eps>0$. Then, there is a function $\chi_{s,t}:\proj^k\to\Real$ with $0\leq\chi_{s,t}\leq 1$, such that $\chi_{s,t}\equiv 1$ on $V_t\setminus {E_s}$, $\supp (\chi_{s,t})\subseteq V_{t^{(\delta+1/2)^{-1}}}\setminus E_{\frac{2}{3}s}$, and $\norm{\chi_{s,t}}_\dsh\leq c_\chi\max\set{1, 9s^{-2}}$, where $c_\chi>0$ is a constant independent of $s$, $t$, and $\eps$, $i$, $n$ in $s_{n,i}$, $t_{n,i}$.
\end{lem}

\begin{proof}
We choose a smooth convex increasing function $\chi:\Real\cup\set{-\infty}\to[0, \infty)$ such that $\chi(x)=0$ on $[-\infty, -1]$ and $\chi(x)=x$ for $x\geq 1$ as in Lemma ~\ref{lem:cvxincftn}. Then, $\max\set{x, 0}\leq\chi(x)\leq\max\set{x, 0}+1$. We define:
\begin{eqnarray*}
&&\phi_{s,t}:=-\chi(\varphi_V-\log t-A-1),\\
&&\tilde{\chi}_{s,t}:=\chi(\phi_{s,t}+1)\,\, \textrm{ and}\\
&&\chi_{s,t}:=\chi(2(\tilde{\chi}_{s,t}-\chi^{E_{\frac{2}{3}s}, \frac{1}{3}s})-1),
\end{eqnarray*}
where $\varphi_V$ is the q-psh function in Lemma \ref{lem:lojasiewicz} with our $V$, $\delta$ and $E_V=V\cap E$, and $\chi^{E_{\frac{2}{3}s}, \frac{1}{3}s}$ is the function in Lemma ~\ref{lem:smoothing} with $K=E_{\frac{2}{3}s}$ and $\theta=\frac{1}{3}s$.

For $z\in V_t$, from Lemma \ref{lem:lojasiewicz}, we have
\begin{displaymath}
\varphi_V(z)-\log t-A-1\leq\log \dist(z, V)+A-\log t-A-1\leq -1
\end{displaymath}
So, $\phi_{s,t}=0$ on $V_t$. Since the support of $\chi^{E_{\frac{2}{3}s}, \frac{1}{3}s}$ is in ${E_s}$, we have $\chi_{s,t}=1$ over $V_t\setminus {E_s}$.

For $z\in\proj^k\setminus( V_{t^{(\delta+1/2)^{-1}}}\cup E_{\frac{2}{3}s})$, from Lemma \ref{lem:lojasiewicz}, we have
\begin{eqnarray*}
&&\varphi_V-\log t-A-1\\
&&\geq\delta\log \dist(z, V)+C\log \dist(z, V\cap E)-A-\log t-A-1\\
&&\geq-\frac{1}{2\delta+1}\log (2{t_{n,i}}^{(\delta+1/2)^{-k}})+C\log \Big(\frac{1}{2}(\frac{2}{3})^ks_{n,i}\Big)+C\log\frac{2}{3}-2A-1\\
&&\geq-C(40k^2\delta)^{4k}\log\eps-\frac{1}{2\delta+1}\log 2+C\log\frac{1}{2}\Big(\frac{2}{3}\Big)^{k+1}-2A-1\geq 2
\end{eqnarray*}
since $\eps$ is sufficiently small. So, we have $\phi_{s,t}\leq -2$ and $\tilde{\chi}_{s,t}=0$ on $\proj^k\setminus( V_{t^{(\delta+1/2)^{-1}}}\cup E_{\frac{2}{3}s})$. Since $0\leq\tilde{\chi}_{s,t}\leq 1$ and $\chi^{E_{\frac{2}{3}s}, \frac{1}{3}s}=1$ over $E_{\frac{2}{3}s}$, we have $\chi_{s,t}=0$ outside $V_{t^{(\delta+1/2)^{-1}}}\setminus E_{\frac{2}{3}s}$. This proves the support part of the statement.

Since $0\leq\chi_{s,t}\leq 1$, the $\cali{L}^1$-norm of $\chi_{s,t}$ is uniformly bounded with respect to $s$, $t$, and $\eps$, $i$, $n$ in $s_{n,i}$, $t_{n,i}$. Thus, Lemma ~\ref{lem:cvxincftn} together with $\varphi$ being q-psh and $\norm{\chi^{E_{\frac{2}{3}s}, \frac{1}{3}s}}_{\cali{C}^2}\lesssim 9s^{-2}$ from Lemma \ref{lem:smoothing} prove that the DSH norm is bounded by $c_\chi\max\set{1, 9s^{-2}}$ for some constant $c_\chi>0$ independent of $s$, $t$, and $\eps$, $i$, $n$ in $s_{n,i}$, $t_{n,i}$. This completes the proof.
\end{proof}

We discuss the H\"{o}lder continuity of the quasi-potentials of $f_*(\omega)$ outside $E$.
\begin{defn}
Let $H, \nu>0$. A continuous map $g:\proj^k\to\Real$ is said to be $(H, \nu)$-H\"{o}lder continuous in a subset $U\subset\proj^k$ if every pair $x, y\in U$ satisfies
\begin{displaymath}
\abs{g(x)-g(y)}\leq H \dist(x,y)^\nu.
\end{displaymath}
\end{defn}

\begin{lem}\label{lem:Holder} Assume that $s$ is sufficiently small. The quasi-potentials of $f_*(\omega)$ are $(c_Hs^{-N_4}, \delta^{-1})$-H\"{o}lder continuous in ${E_s}^c$ with $c_H>0$ a constant independent of $s$, where $N_4$ is the constant in Corollary \ref{cor:lojasiewicz}.
\end{lem}

\begin{proof}

The H\"{o}lder continuity of the quasi-potentials of $f_*(\omega)$ over $\proj^k$ is directly from Lemma 5.4.8 in \cite{superpotentials}.

We now prove the $(c_Hs^{-N_4}, \delta^{-1})$-H\"{o}lder continuity of $v$ in ${E_s}^c$. Let $v$ be a quasi-potential of $f_*(\omega)$. Let $x\in{E_s}^c\subset\proj^k$. Consider a ball $B_x$ centered at $x$ and of radius 
$\frac{s}{2}$. If $y\notin B_x$, then
\begin{displaymath}
\abs{v(x)-v(y)}\leq 2\max_{\proj^k}v\leq \frac{2\max_{\proj^k}v}{(s/2)^{\delta^{-1}}}\dist(x, y)^{\delta^{-1}}.
\end{displaymath}
We consider the case where $y\in B_x$. On $B_x$, $v$ can be written as 
\begin{displaymath}
v(z)=\sum_{w\in f^{-1}(z)}u(w)\,\,\,\,\,\textrm{ for }z\in B_x,
\end{displaymath}
where $u$ is a smooth quasi-potential of $\omega$ and the points in $f^{-1}(z)$ are repeated according to their multiplicity.

By Corollary ~\ref{cor:lojasiewicz}, we can write, for $z$ and $z'$ in $B_x$,
\begin{displaymath}
f^{-1}(z)=\set{w^1, \cdots, w^{d^k}}\,\,\,\,\,\textrm{ and }\,\,\,\,\,f^{-1}(z')=\set{{w^1}', \cdots, {w^{d^k}}'},
\end{displaymath}
so that $\dist(w^j, {w^j}')\lesssim (\frac{s}{2})^{-N_4}\dist(z, z')^{\delta^{-1}}$. Hence,
\begin{displaymath}
\abs{v(x)-v(y)}\lesssim d^k\norm{u}_{\cali{C}^1}\max_{1\leq i\leq d^k} \dist(x^i, y^i)\lesssim s^{-N_4} \dist(x, y)^{\delta^{-1}},
\end{displaymath}
where $f(x^i)=x$ and $f(y^i)=y$ for $i=1, ..., d^k$ as above.
This
proves the lemma. The constant $c_H$ being independent of $s$ is clear from our argument. 
\end{proof}

Proposition \ref{prop:estimate_off_E} below is a modification of Proposition 2.3.6 in ~\cite{superpotentials}. Since we can apply the same proof of Proposition 2.3.6 in ~\cite{superpotentials} except Lemma 2.3.8, it suffices to modify Lemma 2.3.8 to Lemma \ref{lem:estimate_off_E} and prove Lemma \ref{lem:estimate_off_E} in order to show Proposition \ref{prop:estimate_off_E}. For convenience, we quote Lemma 2.3.7, Lemma 2.3.9 and Lemma 2.3.10 of ~\cite{superpotentials} after the proof of Lemma \ref{lem:estimate_off_E}.
\begin{prop}\label{prop:estimate_off_E} 
Let $s, t$ be two positive real numbers such that $\frac{1}{2}s_{n,i}\leq s\leq 2s_{n,i}$ and $\frac{1}{2}t_{n,i}\leq t\leq 2t_{n,i}$ for some $s_{n,i}, t_{n,i}$ with sufficiently small $\eps>0$. Let $R\in\cali{C}_p$ and $U$ the Green quasi-potential of a current $R$. Then, we have:
\begin{displaymath}
\abs{\int_{V_t\setminus({E_s}\cup V)}U\wedge f_*(\omega)^{k-p+1}}\leq c_Vt^\beta s^{-2-N_4},
\end{displaymath}
where $N_4$ is the constant in Corollary ~\ref{cor:lojasiewicz}, $\beta:=(20k^2(\delta+1/2))^{-k}\delta^{-k+p-1}$, and $c_V>0$ is a constant independent of $s$, $t$, $R$ and $\eps$, $i$, $n$ in $s_{n,i}$, $t_{n,i}$.
\end{prop}

Consider $M>0$ and define $\eta_M:=\min\set{0, M+\eta}$ where $\eta$ is from Proposition \ref{prop:kernel}. We have $\norm{\eta_M}_\dsh$ is uniformly bounded with respect to $M$
and $\eta_M-M\leq\eta$. Define $K_M:=-M\Theta$ and $K_M':=\eta_M\Theta$ where $\Theta$ is from Proposition \ref{prop:kernel}. Then, $K_M$ is negative closed and we have $K_M+K_M'\leq K$. Define
\begin{displaymath}
U_M(z):=\int_\xi R(\xi)\wedge K_M(z, \xi)\,\,\,\textrm{ and }\,\,\,U_M'(z):=\int_\xi R(\xi)\wedge K_M'(z, \xi).
\end{displaymath}
The form $U_M$ is negative closed and of mass $\simeq M$ and $U_M+U_M'\leq U$. Choose $M:=t^{-\beta}$. Note that $U$ is negative and that $\Theta$ has singularities of order $\dist(z, \xi)^{2-2k}$.
\medskip

We modify Lemma 2.3.8 in \cite{superpotentials}. From here to Lemma \ref{lem:2310}, the inequalities $\lesssim$ and $\gtrsim$ are up to a constant multiple independent of $s$, $t$, $R$ and $\eps$, $i$, $n$ in $s_{n,i}$, $t_{n,i}$. By continuity,
 we may assume that $R$ and $U$ are smooth. We also have that $U_M$ is smooth.

\begin{lem}\label{lem:estimate_off_E} 
Let $l$ be an integer such that $0\leq l\leq k-p+1$. Let $s, t$ be two positive real numbers such that $\frac{1}{2}(\frac{3}{2})^{k-p-l+1}s_{n,i}\leq s\leq 2s_{n,i}$ and $\frac{1}{2}t_{n,i}\leq t\leq 2(t_{n,i})^{(\frac{1}{\delta+1/2})^{k-p-l+1}}$ for some $s_{n,i}, t_{n,i}$ with sufficiently small $\eps>0$. Then, we have
\begin{displaymath}
\abs{\int_{V_t\setminus ({E_s}\cup V)}U_M\wedge(\frac{1}{d^{k-1}}f_*(\omega))^l\wedge\omega^{k-p-l+1}}\lesssim t^{\beta_l}s^{-2-N_4},
\end{displaymath}
where $\beta_l=(20k^2(\delta+1/2))^{-l}\delta^{-l}$.
\end{lem}

\begin{proof} We use an induction argument on $l$. The case $l=0$ follows from Lemma 2.3.7 in \cite{superpotentials} (or see Lemma \ref{lem:237}). Assume that the estimate is true for $l-1$. Take $\chi_{s,t}$ in Lemma \ref{lem:cut-off}. Then,

\begin{eqnarray*}
&&-\int_{V_{t}\setminus E_{s}}U_M\wedge (\frac{1}{d^{k-1}}f_*(\omega))^l\wedge\omega^{k-p-l+1}\\
&&\leq-\int\chi_{s,t} U_M\wedge (\frac{1}{d^{k-1}}f_*(\omega))^l\wedge\omega^{k-p-l+1}\\
&&=\int-\chi_{s,t}U_M\wedge (\frac{1}{d^{k-1}}f_*(\omega))^{l-1}\wedge\omega^{k-p-l+2}\\
&&\hspace{.5in}+\int-\chi_{s,t}U_M\wedge (\frac{1}{d^{k-1}}f_*(\omega))^{l-1}\wedge dd^cu\wedge\omega^{k-p-l+1},
\end{eqnarray*}
where $\frac{1}{d^{k-1}}f_*(\omega)=\omega+dd^cu$ and $\sup_{\proj^k}u=0$. The bound of the first integral is from the induction hypothesis with the replacement of $t$ by ${t}^{\frac{1}{\delta+1/2}}$ and $s$ by $\frac{2}{3}s$.

We compute the bound of the second integral. Lemma ~\ref{lem:Holder} implies that $u$ is $(c_H(\frac{s}{2})^{-N_4}, \delta^{-1})$-H\"{o}lder continuous in ${E_\frac{s}{2}}^c$. By Proposition ~\ref{prop:216} and Lemma ~\ref{lem:218}, for small ${\eps_\mathrm{reg}}>0$ so that $\eps_\mathrm{reg}\ll s$, we can find a smooth function $u_{\eps_\mathrm{reg}}$ by regularizing $u$ by automorphisms of $\proj^k$. In the remaining of the proof, the constants related to the inequalities $\lesssim$ and $\gtrsim$ are also independent of $\eps_{\mathrm{reg}}$. Proposition ~\ref{prop:216} gives $\norm{u_{\eps_\mathrm{reg}}}_{\cali{C}^2}\lesssim {\eps_\mathrm{reg}}^{-2k^2-4k-2}$. 
Due to the H\"{o}lder continuity of $u$ in ${E_\frac{s}{2}}^c$ and Lemma \ref{lem:218}, we have
\begin{displaymath}
\norm{u-u_{\eps_\mathrm{reg}}}_{\infty, {E_\frac{2s}{3}}^c}\lesssim s^{-N_4}{\eps_\mathrm{reg}}^{\delta^{-1}}.
\end{displaymath}
Here, ${\eps_\mathrm{reg}}$ will be determined later. Then, the second integral can be split as follows:
\begin{eqnarray*}
&&\int-\chi_{s,t}U_M\wedge (\frac{1}{d^{k-1}}f_*(\omega))^{l-1}\wedge dd^cu\wedge\omega^{k-p-l+1}\\
&&= \int-\chi_{s,t}U_M\wedge (\frac{1}{d^{k-1}}f_*(\omega))^{l-1}\wedge dd^cu_{\eps_\mathrm{reg}}\wedge\omega^{k-p-l+1}\\
&&\hspace{.5in}+\int-(u-u_{\eps_\mathrm{reg}})dd^c\chi_{s,t}\wedge U_M\wedge (\frac{1}{d^{k-1}}f_*(\omega))^{l-1}\wedge\omega^{k-p-l+1}.
\end{eqnarray*}
From $\norm{u_{\eps_\mathrm{reg}}}_{\cali{C}^2}$ and the induction hypothesis together with the replacement of $t$ by ${t}^{\frac{1}{\delta+1/2}}$ and $s$ by $\frac{2}{3}s$, the first integral $\lesssim{\eps_\mathrm{reg}}^{-2k^2-4k-2}{t}^{\delta^{-1}\beta_{l-1}}{s}^{-2-N_4}$. 
The bound of the second integral is computed cohomologically. From the DSH-norm of $\chi_{s,t}$ and the mass of $U_M$ being of the same order of $M={t}^{-\beta}$, we have the second integral $\lesssim s^{-N_4}{\eps_\mathrm{reg}}^{\delta^{-1}}s^{-2}t^{-\beta}$. Take ${\eps_\mathrm{reg}}:={t}^{\delta^{-1}(2k^2+4k+2+\delta^{-1})^{-1}\beta_{l-1}}$. Note that due to the difference of the shrinking speed of $t_{n, i}$ and $s_{n, i}$ as $\eps\to 0$, we can always find such $\eps_{\mathrm{reg}}$ satisfying the above discussion for all sufficiently small $\eps>0$. Then,
\begin{displaymath}
1-\frac{2k^2+4k+2}{2k^2+4k+2+\delta^{-1}}\geq\frac{\delta^{-1}}{10k^2}.
\end{displaymath}
In result, we have
\begin{eqnarray*}
&&{t}^{\delta^{-1}\beta_{l-1}}{\eps_\mathrm{reg}}^{-2k^2-4k-2}\leq {t}^{\delta^{-1}\beta_{l-1}(10k^2)^{-1}\delta^{-1}}\leq {t}^{\beta_l}\\
&&{t}^{-\beta}{\eps_\mathrm{reg}}^{\delta^{-1}}\leq{t}^{-\beta}{t}^{(10k^2\delta)^{-1}\beta_{l-1}\delta^{-1}}\leq{t}^{-\beta}{t}^{2\beta_l}\leq{t}^{\beta_l}.
\end{eqnarray*}
This completes the induction and proves the lemma.
\end{proof}

Assume that $t$ is sufficiently small as in Proposition \ref{prop:estimate_off_E}.
\begin{lem}[See Lemma 2.3.7 in \cite{superpotentials}]\label{lem:237} $\abs{\int_{V_t}U_M\wedge\omega^{k-p+1}}\lesssim t$.
\end{lem}

\begin{lem}[See Lemma 2.3.9 in \cite{superpotentials}]\label{lem:239} $\norm{U'_M}\lesssim e^{-\frac{1}{2}M}$.
\end{lem}

\begin{lem}[See Lemma 2.3.10 in \cite{superpotentials}]\label{lem:2310} For every integer $l$ with $0\leq l\leq k-p+1$, we have $\abs{\int U'_M\wedge f_*(\omega)^l\wedge\omega^{k-p-l+1}}\lesssim e^{-\frac{1}{2}(10k^2)^{-l}d^{-kl}M}$.
\end{lem}
Notice that $e^{-\frac{1}{2}(10k^2)^{-l}d^{-kl}M}<e^{-\frac{1}{2}(10k^2)^{-k}d^{-k^2}M}<M^{-1}=t^\beta$ for a sufficiently small $t>0$ and for every integer $l$ with $0\leq l\leq k-p+1$.
\medskip

In the same way as in the proof of Lemma ~\ref{lem:estimate_off_E}, we obtain the following lemma.
\begin{lem}\label{lem:estimate_off_E_omega}
Let $l$ be an integer such that $0\leq l\leq k-p+1$. Let $s, t$ be two positive real numbers as in Lemma ~\ref{lem:estimate_off_E}. Then,
\begin{displaymath}
\abs{\int_{V_t\setminus({E_s}\cup V)}\omega^{k-l}\wedge f_*(\omega)^l}\lesssim t^{\beta_l}s^{-2-N_4},
\end{displaymath}
where $\beta_l=(20k^2(\delta+1/2))^{-l}\delta^{-l}$.
\end{lem}

Also, concerning $R$ and $U$ in Proposition \ref{prop:estimate_off_E}, the proofs of Lemma \ref{lem:estimate_off_E} and Lemma \ref{lem:2310} use only the mass of $U_M$ and the DSH norm of $\eta_M$. Note that $\norm{\tau_\zeta}_{\cali{C}^2}$ is uniformly bounded for $\norm{\zeta}_A<1$. So, we can apply the same proofs of Lemma \ref{lem:estimate_off_E} and Lemma \ref{lem:2310} to $(\tau_\zeta)_*(U_M)$ and $(\tau_\zeta)_*(U_M')$, respectively for $\norm{\zeta}_A<1$. Lemma \ref{lem:237} and Lemma \ref{lem:239} are still true when $U_M$ and $U_M'$ are replaced by $(\tau_\zeta)_*(U_M)$ and $(\tau_\zeta)_*(U_M')$, respectively for $\norm{\zeta}_A<1$ as well. Hence, we obtain the following:
\begin{prop}\label{prop:estimate_off_EAut}
Let $s, t$ be two positive real numbers as in Proposition \ref{prop:estimate_off_E}. Let $R\in\cali{C}_p$ and $U$ the Green quasi-potential of a current $R$. Then, we have:
\begin{displaymath}
\abs{\int_{V_t\setminus({E_s}\cup V)}(\tau_\zeta)_*(U)\wedge f_*(\omega)^{k-p+1}}\leq c_Vt^\beta s^{-2-N_4},
\end{displaymath}
where $N_4$ is the constant in Corollary ~\ref{cor:lojasiewicz}, $\beta:=(20k^2(\delta+1/2))^{-k}\delta^{-k+p-1}$, and $c_V>0$ is a constant independent of $s$, $t$, $R$ and $\eps$, $i$, $n$ in $s_{n,i}$, $t_{n,i}$ and $\zeta$ with $\norm{\zeta}_A<1$.
\end{prop}



Now we prove the first one of our claims for the proof of Proposiiton \ref{prop:mainthm2} using Proposition \ref{prop:estimate_off_E}. The point is that $t_{n, i}\ll s_{n, i}$ and $t_{n, i}\ll \eps_{n, i-1}$. In the statement and the proof of Lemma \ref{lem:1st_estimate}, the inequalities $\lesssim$ and $\gtrsim$ are up to a constant multiple independent of $\eps, i, n$ and $R$.
\begin{lem}\label{lem:1st_estimate}Let $\eps>0$ be sufficiently small. Then, 
\begin{displaymath}
\agl{V_{n-i}', \Lambda(R_{n,i-1})-R_{n,i}}_{W_{n,i}\setminus V}\gtrsim-\eps^{ni}
\end{displaymath}
\end{lem}

\begin{proof}[The proof of Lemma \ref{lem:1st_estimate}]
Since $V_{n-i}'$ is negative, $\agl{V_{n-i}', \Lambda(R_{n,i-1})-R_{n,i}}_{W_{n,i}\setminus V}\geq\agl{V_{n-i}', \Lambda(R_{n,i-1})}_{W_{n,i}\setminus V}$. Thus, it suffices to prove $\abs{\agl{V_{n-i}', \Lambda(R_{n,i-1})}_{W_{n,i}\setminus V}}\lesssim\eps^{ni}$.

\begin{eqnarray*}
\abs{\agl{V_{n-i}', \Lambda(R_{n,i-1})}_{W_{n,i}\setminus V}}&\lesssim&\norm{R_{n, i-1}}_\infty\abs{\agl{V_{n-i}', \Lambda(\omega^{k-p+1})}_{W_{n,i}\setminus V}}\\
&\lesssim&\norm{R_{n, i-1}}_\infty\abs{\agl{V_{n-i}', f_*(\omega)^{k-p+1}}_{W_{n,i}\setminus V}}\\
&\lesssim&{\eps_{n,i-1}}^{-4k^2}{t_{n,i}}^\beta {s_{n,i}}^{-2-N_4}\\
&\lesssim&\eps^{n[C(2+N_4)(40k^2\delta)^{6k(i-1)}(\frac{(40k^2\delta)^{6k}}{10k(40k^2\delta)^k\delta^k}-4k^2)-i]}\lesssim\eps^{ni},
\end{eqnarray*}
where $C>1$ is in Lemma ~\ref{lem:lojasiewicz}. In the above, the second inequality is from the current inequality $f_*(\omega^p)\leq (f_*(\omega))^p$ and the third inequality from Proposition ~\ref{prop:216}, Proposition ~\ref{prop:estimate_off_E}, and Lemma ~\ref{lem:estimate_off_E_omega}.
\end{proof}

\section{The Proof of $\agl{V_{n-i}', \Lambda(R_{n,i-1})-R_{n,i}}_{W_{n,i}'}\gtrsim-\eps^{ni}$}\label{sec:3}
We prove the second one of our claims for the proof of Proposition \ref{prop:mainthm2} in the same way as in Lemma 5.4.7 in ~\cite{superpotentials}. The following is from Lemma 5.4.7 in \cite{superpotentials}.

\begin{lem}\label{lem:Calpha_estimate_Wi2}
Let $\alpha\geq 0$. For all sufficiently small $t>0$, we have
\begin{displaymath}
\norm{\Lambda(R)}_{\cali{C}^\alpha({V_{t}}^c)}\lesssim\norm{R}_{\cali{C}^\alpha}t^{-(4+\alpha)k},
\end{displaymath}
for any smooth form $R$ of bidegree $(p, p)$ with $0\leq p\leq k$. The inequality is up to a constant multiple independent of $t$ and $R$.
\end{lem}

\begin{proof}
First, we fix a finite atlas for $\proj^k$. 
%
%
Without loss of generality, we may assume that $t>0$ is sufficiently small.
Consider $a\in {V_t}^c$. We denote by $B_a(t/2)$ the ball centered at $a\in V_t^c$ and of radius $t/2$. Then, there exist $d^k$ injective holomorphic maps $g_j:B_a(t/2)\to\proj^k$ such that $f\circ g_j=\id$ on $B_a(t/2)$. 
Such $g_j$'s are uniformly bounded by a uniform constant with respect to the finite atlas.
In particular, the constant is independent of $t$ and $a$. From the Cauchy-integral formula, we know that all the derivatives of order $n$ of $g_j$ on $B_{t/4}$ are $\lesssim t^{-n}$. On $B_a(t/2)$, we have
\begin{displaymath}
\Lambda(R)=d^{-p+1}\sum^{d^k}_{j=1}g_j^*(R).
\end{displaymath}
For fixed local real coordinates $(x_1, ..., x_{2k})$, $R$ is a combination with smooth coefficients of $dx_{j_1}\wedge\cdots\wedge dx_{j_{2k-2p+2}}$. Hence, the estimate on the derivatives of $g_j$ implies that
\begin{displaymath}
\norm{g_j^*(R)}_{\cali{C}^\alpha(B_{t/4})}\lesssim\norm{R}_{\cali{C}^\alpha}t^{-2k+2p-2-\alpha}\lesssim\norm{R}_{\cali{C}^\alpha}t^{-(4+\alpha)k}.
\end{displaymath}
It follows that
\begin{displaymath}
\norm{\Lambda(R)}_{\cali{C}^\alpha({V_{(3/4)t}}^c)}\lesssim\norm{R}_{\cali{C}^\alpha}t^{-(4+\alpha)k}.
\end{displaymath}
This implies the lemma.
\end{proof}

We prove the second one of our claims for the proof of Proposition \ref{prop:mainthm2}. In the following lemma and its proof, the inequalities $\lesssim$ and $\gtrsim$ are up to a constant multiple independent of $\eps, i, n$ and $R$.
\begin{lem}\label{lem:3rd_estimate} Let $\eps>0$ be sufficiently small. Then,
\begin{displaymath}
\abs{\langle V_{n-i}', \Lambda(R_{n,i-1})-R_{n,i}\rangle_{W'_{n,i}}} \lesssim \eps^{ni}
\end{displaymath}
\end{lem}
\begin{proof}

Let $\tau_{{\eps_{n,i}}\zeta}$ be an automorphism in $\Aut(\proj^k)$ such that $\norm{\zeta}_A<1$. Since $(t_{n,i})^{10k}=\eps_{n,i}$, we may assume that $\tau_{{\eps_{n,i}}\zeta}(W_{n,i}')\subseteq {V_{{t_{n,i}}/2}}^c$ for all $\zeta$ with $\norm{\zeta}_A<1$ and for all sufficiently small $\eps>0$. With Lemma ~\ref{lem:Calpha_estimate_Wi2} for $\alpha=1$ and Lemma ~\ref{lem:218}, we have 
\begin{eqnarray*}
&&\norm{(\tau_{{\eps_{n,i}}\zeta})_*(\Lambda(R_{n,i-1}))-\Lambda(R_{n,i-1})}_{\infty,W_{n,i}'}\\
&&\lesssim\norm{\Lambda(R_{n,i-1})}_{\cali{C}^1({V_{{t_{n,i}}/2}}^c)}\dist(\tau_{{\eps_{n,i}}\zeta}, \id)
\lesssim\norm{R_{n,i-1}}_{\cali{C}^1}{t_{n,i}}^{-5k}{\eps_{n,i}}.
\end{eqnarray*}

From Proposition ~\ref{prop:216}, $\norm{R_{n,i-1}}_{{\C}^1}\lesssim {\eps_{n,i-1}}^{-2k^2-4k-1} \lesssim {\eps_{n,i-1}}^{-5k^2}$. So, we have $\norm{\Lambda(R_{n,i-1})-R_{n,i}}_{\infty,W_{n,i}'}\lesssim {t_{n,i}}^{-5k}{\eps_{n,i}}{\eps_{n,i-1}}^{-5k^2}$. Since $V_{n-i}'$ is negative and has bounded mass independent of $n$ and $i$, we have
\begin{displaymath}
\abs{\langle V_{n-i}', \Lambda(R_{n,i-1})-R_{n,i}\rangle_{W'_{n,i}}} \lesssim {t_{n,i}}^{-5k}{\eps_{n,i}}{\eps_{n,i-1}}^{-5k^2}={\eps_{n,i}}^\frac{1}{2}{\eps_{n,i-1}}^{-5k^2}\leq \eps^{ni}.
\end{displaymath}
\end{proof}

\section{The Proof of $\agl{V'_{n-i}, \Lambda(R_{n,i-1})-R_{n,i}}_{W_{n,i}''\setminus V}\gtrsim-\eps^{ni}$}\label{sec:2}
In this section, we prove Lemma \ref{lem:2nd_estimate} below, the third one of our claims for the proof of Proposition \ref{prop:mainthm2}. In the rest of the paper, the inequalities $\lesssim$ and $\gtrsim$ are up to a constant multiple independent of $\eps, i, n$ and $R$.

\begin{lem}\label{lem:2nd_estimate} Let $\eps>0$ be sufficiently small. Then,
\begin{displaymath}
\abs{\agl{V_{n-i}',\Lambda(R_{n,i-1})-R_{n,i}}_{W_{n,i}''\setminus V}}\lesssim(1+\norm{S}_{\infty, E_{\eps}})\eps^{ni}
\end{displaymath}
\end{lem}

For this estimate, we will use change of coordinates and Fubini's theorem to find an upper bound for $\abs{\agl{V_{n-i}', \Lambda(R_{n,i-1})-R_{n,i}}_{W_{n,i}''\setminus V}}$ in terms of the mass of $\Lambda(R_{n,i-1})$ and the pointwise difference of a smooth form and its regularization by automorphisms.

\begin{proof}
In order to handle boundary of $W_{n,i}''$, we consider three neighborhoods of $E$. Notice that $s_{n,i}\gg\eps_{n,i}$. Let $C_\Aut>0$ be a constant such that for $x\in\proj^k$, for $\zeta$ with $\norm{\zeta}_A<1$ and for all sufficiently small $\epsilon>0$, $\dist(\tau_{\epsilon \zeta}(x),x)< C_{\Aut}\epsilon$ is true. Define
\begin{itemize}
\item $W_{n,i,0}'':=$ the $(s_{n,i}-3C_\Aut\eps_{n,i})$-neighborhood of $E$,

\item $W_{n,i,1}'':=$ the $(s_{n,i}-C_\Aut\eps_{n,i})$-neighborhood of $E$, and

\item $W_{n,i,2}':=$ the $(s_{n,i}+C_\Aut\eps_{n,i})$-neighborhood of $E$.

\end{itemize}

We estimate $\agl{V_{n-i}',\Lambda(R_{n,i-1})-R_{n,i}}_{W_{n,i}''\setminus V}$. Note that $V_{n-i}'$ is smooth in $W_{n,i,2}''$.
\begin{eqnarray*}
&&\agl{V_{n-i}',\Lambda(R_{n,i-1})-R_{n,i}}_{W_{n,i}''\setminus V}=\agl{V_{n-i}',\Lambda(R_{n,i-1})-R_{n,i}}_{W_{n,i}''}\\
&&=\int_{W_{n,i}''}V_{n-i}'\wedge\Lambda(R_{n,i-1})-\int_{W_{n,i}''}\int_{\set{\norm{\zeta}_A<1}}V_{n-i}'\wedge(\tau_{\eps_{n,i}\zeta})^*\Lambda(R_{n,i-1})d\rho(\zeta).
\end{eqnarray*}

Since $\Lambda(R_{n,i-1})$ has $\cali{L}^1$-coefficients, we use Fubini's Theorem:
\begin{equation*}
=\int_{W_{n,i}''}V_{n-i}'\wedge\Lambda(R_{n,i-1})-\int_{\set{\norm{\zeta}_A<1}}\int_{W_{n,i}''}V_{n-i}'\wedge(\tau_{\eps_{n,i}\zeta})^*\Lambda(R_{n,i-1})d\rho(\zeta).
\end{equation*}

Since each $\tau_{\eps_{n,i}\zeta}$ is an automorphism, by use of change of coordinates, we have
\begin{displaymath}
=\int_{W_{n,i}''}V_{n-i}'\wedge\Lambda(R_{n,i-1})-\int_{\set{\norm{\zeta}_A<1}}\int_{\tau_{\eps_{n,i}\zeta}(W_{n,i}'')}(\tau_{\eps_{n,i}\zeta})_*(V_{n-i}')\wedge\Lambda(R_{n,i-1})d\rho(\zeta).\\
\end{displaymath}

We have $W_{n,i,1}''\subseteq W_{n,i}''$ and $W_{n,i,1}''\subseteq \tau_{\eps_{n,i}\zeta}(W_{n,i}'')$ for all $\zeta$ with $\norm{\zeta}_A<1$. Then, the last integral can be written as the following sum:

\begin{eqnarray}
&&\label{eqn:1st}\int_{W_{n,i,1}''}V_{n-i}'\wedge\Lambda(R_{n,i-1})-\int_{\set{\norm{\zeta}_A<1}}\int_{W_{n,i,1}''}(\tau_{\eps_{n,i}\zeta})_*(V_{n-i}')\wedge\Lambda(R_{n,i-1})d\rho(\zeta)\\
&&\label{eqn:2nd}\hspace{.5in}+\int_{W_{n,i}''\setminus W_{n,i,1}''}V_{n-i}'\wedge\Lambda(R_{n,i-1})\\
&&\label{eqn:3rd}\hspace{.5in}-\int_{\set{\norm{\zeta}_A<1}}\int_{\tau_{\eps_{n,i}\zeta}(W_{n,i}'')\setminus W_{n,i,1}''}(\tau_{\eps_{n,i}\zeta})_*(V_{n-i}')\wedge \Lambda(R_{n,i-1})d\rho(\zeta).
\end{eqnarray}

For these computations, we need the following lemma.
\begin{lem}\label{lem:derivativeWi1}
Let $i,n\in\mathbb{N}$ be integers such that $0\leq i\leq n$ and $\eps>0$ sufficiently small. Then,
\begin{displaymath}
\norm{V_{n-i}'}_{\cali{C}^1(W_{n,i,2}'')}\lesssim {s_{n,i}}^{-2k}(1+\norm{f}_{\cali{C}^1})^{p(n-i)}(1+\norm{S}_{\infty, E_{\eps}}).
\end{displaymath}
\end{lem}

\begin{proof}
From our settings in Section \ref{sec:settings}, $L^{n-i}(S)$ is smooth over $\frac{1}{2\sqrt{\eps}}W_{n,i}''$. Observe that in this estimate, we do not require the closedness of $L^{n-i}(S)$. Consider a cut-off function $\chi^{E_{\frac{1}{4\sqrt{\eps}}s_{n,i}},s_{n,i}}$ in Lemma \ref{lem:smoothing}, denoted by $\chi$ in this proof. Write $L^{n-i}(S)=\chi L^{n-i}(S)+(1-\chi)L^{n-i}(S)$. Then, $\chi L^{n-i}(S)$ is a positive smooth $(p,p)$-form on $\proj^k$, $(1-\chi) L^{n-i}(S)$ is a positive $(p,p)$-current with $\supp((1-\chi)L^{n-i}(S))\cap E_{\frac{1}{4\sqrt{\eps}}s_{n,i}}=\emptyset$ and
\begin{displaymath}
U_{n-i}=\int_{\xi\neq z}[\chi L^{n-i}(S)](\xi)\wedge K(z, \xi)+\int_{\xi\neq z}[(1-\chi) L^{n-i}(S)](\xi)\wedge K(z, \xi).
\end{displaymath}

As in the proof of Lemma 2.3.5 in \cite{superpotentials}, the first term is smooth and we have:
\begin{eqnarray*}
&&\norm{\int_{\xi\neq z}[\chi L^{n-i}(S)](\xi)\wedge K(z, \xi)}_{\cali{C}^1}\\
&&\lesssim\norm{[\chi L^{n-i}(S)](\xi)}_\infty\cdot\norm{\int_{\xi\neq z}\omega^p(\xi)\wedge K(z, \xi)}_{\cali{C}^1}\lesssim{\norm{f}_{\cali{C}^1}}^{p(n-i)}\norm{S}_{\infty, E_\eps}.
\end{eqnarray*}

We consider the other part: $\int_{\xi\neq z}[(1-\chi) L^{n-i}(S)](\xi)\wedge K(z, \xi)$. For notational convenience, we denote $(1-\chi) L^{n-i}(S)$ by $S_{n-i}$ in the rest of the proof. From our settings, we have that $\dist(E, \supp(S_{n-i}))\geq\frac{1}{4\sqrt{\eps}}s_{n,i}$. So, consider $z\in W_{n,i,2}''$. Then, $\dist(z, {O_{n-i}}^c)\geq(\frac{1}{4\sqrt{\eps}}-1-C_\Aut)s_{n,i}$, and
\begin{equation}\label{eqn:1-chi}
\int_{\xi\neq z}S_{n-i}(\xi)\wedge K(z, \xi)=\int_{\xi\not\in B_z((\frac{1}{4\sqrt{\eps}}-1-C_\Aut)s_{n,i})}S_{n-i}(\xi)\wedge K(z, \xi),
\end{equation}
where $B_z((\frac{1}{4\sqrt{\eps}}-1-C_\Aut)s_{n,i})$ is the ball centered at $z$ and of radius $(\frac{1}{4\sqrt{\eps}}-1-C_\Aut)s_{n,i}$. 
For each $\xi\not\in B_z((\frac{1}{4\sqrt{\eps}}-1-C_\Aut)s_{n,i})$, $K(z, \xi)$ is smooth. 
Hence, $\int_{\xi\neq z}S_{n-i}(\xi)\wedge K(z, \xi)$ is smooth.

We estimate the $\cali{C}_1$-norm of $\int_{\xi\neq z}S_{n-i}(\xi)\wedge K(z, \xi)$ in $W_{n,i,2}''$. 
Since $\nabla K$ has singularities of order $(z-\xi)^{1-2k}$ from Proposition \ref{prop:kernel} and the mass of $S_{n-i}$ is bounded by the mass of $L^{n-i}(S)$, which equals $1$, we have
\begin{displaymath}
\norm{\int_{\xi\neq z}S_{n-i}(\xi)\wedge K(z, \xi)}_{\cali{C}^1}\lesssim((\frac{1}{4\sqrt{\eps}}-1-C_\Aut)s_{n,i})^{1-2k}\leq s_{n,i}^{-2k},
\end{displaymath}
from Equality \ref{eqn:1-chi}.

Since $T^{k-p+1}$ has H\"{o}lder continuous super-potentials, $c_{n-i}$ such that $V_{n-i}'=U_{n-i}-c_{n-i}\omega^{p-1}$ is uniformly bounded with respect to $i, n$. Hence, the above estimate and this relationship prove the lemma.
\end{proof}
  
\subsection*{The Computation of ~\ref{eqn:1st}}
From Lemma ~\ref{lem:derivativeWi1}, $V_{n-i}'$ is smooth over $W_{n,i}''$. Then,
\begin{eqnarray*}
&&\abs{~\ref{eqn:1st}}=\abs{\int_{\set{\norm{\zeta}_A<1}}\int_{W_{n,i,1}''}(V_{n-i}'-(\tau_{\eps_{n,i}\zeta})_*(V_{n-i}'))\wedge\Lambda(R_{n,i-1})d\rho(\zeta)}\\
&& \lesssim
\norm{V_{n-i}'-(\tau_{\eps_{n,i}\zeta})_*(V_{n-i}')}_{\infty, W_{n,i,1}''}\int_{\set{\norm{\zeta}_A<1}}\int_{W_{n,i,1}''}\omega^{p-1}\wedge\Lambda(R_{n,i-1})d\rho(\zeta)\\
&&\,\,\,\,\, \lesssim C_\Aut\norm{V_{n-i}'}_{\cali{C}^1(W_{n,i,1}'')}\eps_{n,i}\int_{\set{\norm{\zeta}_A<1}}\int_{W_{n,i,1}''}\omega^{p-1}\wedge\Lambda(R_{n,i-1})d\rho(\zeta)\\
&&\,\,\,\,\,\lesssim C_\Aut\norm{V_{n-i}'}_{\cali{C}^1(W_{n,i,1}'')}\eps_{n,i}\norm{\Lambda(R_{n,i-1})}=C_\Aut\norm{V_{n-i}'}_{\C^1(W_{n,i,1}'')}\eps_{n,i}.
\end{eqnarray*}
From Lemma ~\ref{lem:derivativeWi1},
we have
\begin{displaymath}
\abs{~\ref{eqn:1st}}\lesssim {s_{n,i}}^{-2k}(1+\norm{f}_{\cali{C}^1})^{p(n-i)}(1+\norm{S}_{\infty, E_{\eps}})\eps_{n,i}.
\end{displaymath}

\subsection*{The Computation of ~\ref{eqn:2nd}}
We split the integral into two integrals as follows:
\begin{eqnarray*}
&&\int_{W_{n,i}''\setminus W_{n,i,1}''}V_{n-i}'\wedge\Lambda(R_{n,i-1})\\
&&=\int_{(W_{n,i}''\setminus W_{n,i,1}'')\cap V_{t_{n,i}}}V_{n-i}'\wedge\Lambda(R_{n,i-1})+\int_{(W_{n,i}''\setminus W_{n,i,1}'')\setminus V_{t_{n,i}}}V_{n-i}'\wedge\Lambda(R_{n,i-1}).
\end{eqnarray*}
We consider the first integral. Since $V_{n-i}'$ is negative and $\Lambda(R_{n,i-1})$ is positive,
\begin{eqnarray*}
&&0\geq\int_{(W_{n,i}''\setminus W_{n,i,1}'')\cap V_{t_{n,i}}}V_{n-i}'\wedge\Lambda(R_{n,i-1})\geq\int_{V_{t_{n,i}}\setminus W_{n,i,1}''}V_{n-i}'\wedge\Lambda(R_{n,i-1})\\
&&=\int_{V_{t_{n,i}}\setminus W_{n,i,1}''}U_{n-i}\wedge\Lambda(R_{n,i-1})-\int_{V_{t_{n,i}}\setminus W_{n,i,1}''}c_{n-i}\omega^{k-p}\wedge\Lambda(R_{n,i-1}).
\end{eqnarray*}
Since $f_*(\omega^p)\leq (f_*(\omega))^p$ in the current sense, the above is:
\begin{eqnarray*}
&&\gtrsim\norm{R_{n,i-1}}_\infty\left(\int_{V_{t_{n,i}}\setminus W_{n,i,1}''}U_{n-i}\wedge\Lambda(\omega^p)-\int_{V_{t_{n,i}}\setminus W_{n,i,1}''}\abs{c_{n-i}}\omega^{k-p}\wedge\Lambda(\omega^p)\right)\\
&&\gtrsim\norm{R_{n,i-1}}_\infty\left(\int_{V_{t_{n,i}}\setminus W_{n,i,1}''}U_{n-i}\wedge f_*(\omega)^p
-\int_{V_{t_{n,i}}\setminus W_{n,i,1}''}\abs{c_{n-i}}\omega^{k-p}\wedge f_*(\omega)^p\right)\\
&&\gtrsim-\norm{R_{n,i-1}}_\infty {t_{n,i}}^\beta {s_{n,i}}^{-2-N_4}\gtrsim-{\eps_{n,i-1}}^{-4k^2}{t_{n,i}}^\beta {s_{n,i}}^{-2-N_4}.
\end{eqnarray*}
The second last inequality is due to Proposition ~\ref{prop:estimate_off_E} and Lemma ~\ref{lem:estimate_off_E_omega} together with $c_{n-i}$ uniformly bounded. The estimate of $\norm{R_{n,i-1}}_\infty$ in the last inequality is from Proposition ~\ref{prop:216}.

We estimate the second integral. Since $(W_{n,i}''\setminus W_{n,i,1}'')\setminus V_{t_{n,i}}\subset W_{n,i}''\cap {V_{t_{n,i}}}^c$, from Lemma ~\ref{lem:derivativeWi1} and Lemma ~\ref{lem:Calpha_estimate_Wi2}, we have
\begin{eqnarray*}
&&\abs{\int_{(W_{n,i}''\setminus W_{n,i,1}'')\setminus V_{t_{n,i}}}V_{n-i}'\wedge\Lambda(R_{n,i-1})}\\
&&\lesssim\norm{V_{n-i}'}_{\infty, W_{n,i}''}\norm{\Lambda(R_{n,i-1})}_{\infty, {V_{t_{n,i}}}^c}\cdot\textrm{ the volume of }(W_{n,i}''\setminus W_{n,i,1}'')\setminus V_{t_{n,i}}\\
&&\lesssim {s_{n,i}}^{-2k}(1+\norm{f}_{\cali{C}^1})^{p(n-i)}(1+\norm{S}_{\infty, E_{\eps}}){\eps_{n,i-1}}^{-4k^2}{t_{n,i}}^{-4k}\eps_{n,i}
\end{eqnarray*}
From the first and the second integrals, since ${t_{n,i}}^{-4k}\eps_{n,i}\leq t^\beta$, we have
\begin{eqnarray*}
\abs{~\ref{eqn:2nd}}
\lesssim {s_{n,i}}^{-2k(2+N_4)}(1+\norm{f}_{\cali{C}^1})^{p(n-i)}(1+\norm{S}_{\infty, E_{\eps}}){\eps_{n,i-1}}^{-4k^2}{t_{n,i}}^\beta.
\end{eqnarray*}

\subsection*{The Computation of ~\ref{eqn:3rd}}
We consider
\begin{displaymath}
\int_{\set{\norm{\zeta}_A<1}}\int_{\tau_{\eps_{n,i}\zeta}(W_{n,i}'')\setminus W_{n,i,1}''}(\tau_{\eps_{n,i}\zeta})_*(V_{n-i}')\wedge \Lambda(R_{n,i-1})d\rho(\zeta).
\end{displaymath}
As in the computation of ~\ref{eqn:2nd}, we split the integral into two parts:
\begin{eqnarray*}
&&=\int_{\set{\norm{\zeta}_A<1}}\int_{(\tau_{\eps_{n,i}\zeta}(W_{n,i}'')\setminus W_{n,i,1}'')\cap V_{t_{n,i}}}(\tau_{\eps_{n,i}\zeta})_*(V_{n-i}')\wedge \Lambda(R_{n,i-1})d\rho(\zeta)\\
&&\hspace{.5in}+\int_{\set{\norm{\zeta}_A<1}}\int_{(\tau_{\eps_{n,i}\zeta}(W_{n,i}'')\setminus W_{n,i,1}'')\setminus V_{t_{n,i}}}(\tau_{\eps_{n,i}\zeta})_*(V_{n-i}')\wedge \Lambda(R_{n,i-1})d\rho(\zeta).
\end{eqnarray*}

We compute the first integral. Since $V_{n-i}'$ is negative and $\Lambda(R_{n,i-1})$ is positive,
\begin{eqnarray*}
&&0\geq\int_{\set{\norm{\zeta}_A<1}}\int_{(\tau_{\eps_{n,i}\zeta}(W_{n,i}'')\setminus W_{n,i,1}'')\cap V_{t_{n,i}}}(\tau_{\eps_{n,i}\zeta})_*(V_{n-i}')\wedge\Lambda(R_{n,i-1})d\rho(\zeta)\\
&&\geq\int_{\set{\norm{\zeta}_A<1}}\int_{V_{t_{n,i}}\setminus W_{n,i,1}''}(\tau_{\eps_{n,i}\zeta})_*(V_{n-i}')\wedge\Lambda(R_{n,i-1})d\rho(\zeta)\\
&&=\int_{\set{\norm{\zeta}_A<1}}\int_{V_{t_{n,i}}\setminus W_{n,i,1}''}(\tau_{\eps_{n,i}\zeta})_*(U_{n-i})\wedge\Lambda(R_{n,i-1})d\rho(\zeta)\\
&&\hspace{.5in}-\int_{\set{\norm{\zeta}_A<1}}\int_{V_{t_{n,i}}\setminus W_{n,i,1}''}(\tau_{\eps_{n,i}\zeta})_*(c_{n-i}\omega^{k-p})\wedge\Lambda(R_{n,i-1})d\rho(\zeta).
\end{eqnarray*}
Since $f_*(\omega^p)\leq (f_*(\omega))^p$ in the current sense, the above is:
\begin{eqnarray*}
&&\gtrsim\norm{R_{n,i-1}}_\infty\left(\int_{\set{\norm{\zeta}_A<1}}\int_{V_{t_{n,i}}\setminus W_{n,i,1}''}(\tau_{\eps_{n,i}\zeta})_*(U_{n-i})\wedge\Lambda(\omega^p)d\rho(\zeta)\right.\\
&&\hspace{.5in}\left.-\int_{\set{\norm{\zeta}_A<1}}\int_{V_{t_{n,i}}\setminus W_{n,i,1}''}(\tau_{\eps_{n,i}\zeta})_*(\abs{c_{n-i}}\omega^{k-p})\wedge\Lambda(\omega^p)d\rho(\zeta)\right)\\
&&\gtrsim\norm{R_{n,i-1}}_\infty\left(\int_{\set{\norm{\zeta}_A<1}}\int_{V_{t_{n,i}}\setminus W_{n,i,1}''}(\tau_{\eps_{n,i}\zeta})_*(U_{n-i})\wedge f_*(\omega)^pd\rho(\zeta)\right.\\
&&\hspace{.5in}\left.-\int_{\set{\norm{\zeta}_A<1}}\int_{V_{t_{n,i}}\setminus W_{n,i,1}''}(\tau_{\eps_{n,i}\zeta})_*(\abs{c_{n-i}}\omega^{k-p})\wedge f_*(\omega)^pd\rho(\zeta)\right)\\
&&\gtrsim-\norm{R_{n,i-1}}_\infty {t_{n,i}}^\beta {s_{n,i}}^{-2-N_4}\gtrsim-{\eps_{n,i-1}}^{-4k^2}{t_{n,i}}^\beta {s_{n,i}}^{-2-N_4}.
\end{eqnarray*}
The second last inequality is due to Proposition ~\ref{prop:estimate_off_EAut} and Lemma ~\ref{lem:estimate_off_E_omega} together with $c_{n-i}$ uniformly bounded. The estimate of $\norm{R_{n,i-1}}_\infty$ in the last inequality is from Proposition ~\ref{prop:216}.

We compute the second integral. Note that $V_{n-i}'$ is negative and $\Lambda(R_{n,i-1})$ is positive. From Lemma ~\ref{lem:218} and $V_{n-i}'$ being smooth, we have for the second integral
\begin{eqnarray*}
&&\abs{\int_{\set{\norm{\zeta}_A<1}}\int_{(\tau_{\eps_{n,i}\zeta}(W_{n,i}'')\setminus W_{n,i,1}'')\setminus V_{t_{n,i}}}(\tau_{\eps_{n,i}\zeta})_*(V_{n-i}')\wedge \Lambda(R_{n,i-1})d\rho(\zeta)}\\
&&\lesssim\int_{\set{\norm{\zeta}_A<1}}\norm{(\tau_{\eps_{n,i}\zeta})_*(V_{n-i}')}_{\infty, \tau_{\eps_{n,i}\zeta}(W_{n,i}'')}\\
&&\hspace{1in}\cdot\int_{(\tau_{\eps_{n,i}\zeta}(W_{n,i}'')\setminus W_{n,i,1}'')\setminus V_{t_{n,i}}}\omega^{k-p}\wedge \Lambda(R_{n,i-1})d\rho(\zeta)\\
&&\lesssim\norm{V_{n-i}'}_{\infty, W_{n,i,2}''}\norm{\Lambda(R_{n,i-1})}_{\infty, V_{t_{n,i}}^c}\cdot\textrm{ the volume of }(W_{n,i,2}''\setminus W_{n,i,1}'')\setminus V_{t_{n,i}}\\
&&\lesssim {s_{n,i}}^{-2k}(1+\norm{f}_{\cali{C}^1})^{p(n-i)}(1+\norm{S}_{\infty, E_{\eps}}){\eps_{n,i-1}}^{-4k^2}{t_{n,i}}^{-4k}\eps_{n,i}.
\end{eqnarray*}
In the second last inequality, we have $\tau_{\eps_{n,i}\zeta}(W_{n,i}'')\subseteq W_{n,i,2}''$ for all $\zeta$ with $\norm{\zeta}_A<1$; in the last inequality, the bound of $\norm{\Lambda(R_{n,i-1})}_{\infty, {V_{t_{n,i}}}^c}$ is from Lemma ~\ref{lem:Calpha_estimate_Wi2} and $\norm{V_{n-i}'}_{\infty, W_{n,i,2}''}$ can be computed in the same way as in Lemma ~\ref{lem:derivativeWi1}.

From the first and the second integrals, since ${t_{n,i}}^{-4k}\eps_{n,i}\leq t^\beta$, we have
\begin{eqnarray*}
\abs{~\ref{eqn:3rd}}
\lesssim {s_{n,i}}^{-2k(2+N_4)}(1+\norm{f}_{\cali{C}^1})^{p(n-i)}(1+\norm{S}_{\infty, E_{\eps}}){\eps_{n,i-1}}^{-4k^2}{t_{n,i}}^\beta.
\end{eqnarray*}

We finish the proof.
From the computations of ~\ref{eqn:1st}, ~\ref{eqn:2nd}, and ~\ref{eqn:3rd}, we have
\begin{eqnarray*}
&&\abs{\agl{V_{n-i}',\Lambda(R_{n,i-1})-R_{n,i}}_{W_{n,i}''\setminus V}}\\
&&\lesssim {s_{n,i}}^{-2k(2+N_4)}(1+\norm{f}_{\cali{C}^1})^{p(n-i)}(1+\norm{S}_{\infty, E_{\eps}}){\eps_{n,i-1}}^{-4k^2}{t_{n,i}}^\beta\\
&&\leq(1+\norm{f}_{\cali{C}^1})^{p(n-i)}(1+\norm{S}_{\infty, E_{\eps}})\\
&&\hspace{.5in}\cdot\eps^{-2nik(2+N_4)}\eps^{-4k^2nC(2+N_4)(40k^2\delta)^{6k(i-1)}}\eps^{nC(2+N_4)\frac{(40k^2\delta)^{6ki}}{10k(40k^2\delta^2)^k}}\\
&&\leq(1+\norm{f}_{\cali{C}^1})^{p(n-i)}(1+\norm{S}_{\infty, E_{\eps}})\eps^{n(2+N_4)[C(40k^2\delta)^{6k(i-1)}(\frac{(40k^2\delta)^{6k}}{10k(40k^2\delta)^k\delta^k}-4k^2)-2ik]}\\
&&\leq (1+\norm{S}_{\infty, E_{\eps}})\eps^{ni}
\end{eqnarray*}
since $\eps>0$ is sufficiently small.
\end{proof}

\section{The Proof of $d^{-n}\V_S(R_{n,n})\gtrsim nd^{-\frac{n}{4}}\log\eps$}\label{sec:4}
In this section, we prove our last claim for the proof of Proposition \ref{prop:mainthm2}. We estimate $d^{-n}\cali{V}_S(R_{n,n})$ in the same way as Lemma 5.4.11 in ~\cite{superpotentials}.

\begin{prop}[See Lemma 3.2.10 in ~\cite{superpotentials}]\label{prop:3210}
Let $W\subset\proj^k$ be an open subset and $K\subset W$ be a compact set. Let $S$ be a positive closed $(p, p)$-current of mass $1$ with support in $K$ and $R$ be a current in $\cali{C}_{k-p+1}$. Assume that the restriction of $R$ to $W$ is a bounded form. Then the super-potential $\cali{U}_S$ of mean $0$ of $S$ satisfies
\begin{displaymath}
\abs{\cali{U}_S(R)}\leq c_\mathrm{bdd}(1+\log^+\norm{R}_{\infty, W}),
\end{displaymath}
where $c_\mathrm{bdd}>0$ is a constant independent of $S$ and $R$, and $\log^+:=\max\set{\log, 0}$.
\end{prop}

We prove the last claim for the proof of Proposition \ref{prop:mainthm2}.
\begin{lem}\label{lem:tail} Let $\eps>0$ be sufficiently small. Then,
\begin{displaymath}
d^{-n}\abs{\V_S(R_{n,n})}\lesssim nd^{-\frac{n}{4}}(-\log\eps).
\end{displaymath}
\end{lem}
\begin{proof}
From Proposition ~\ref{prop:216}, we have $\norm{R_{n,n}}_\infty\lesssim{\eps_{n,n}}^{-4k^2}$. We apply Proposition ~\ref{prop:3210} to $W=K=\proj^k$, $\cali{U}=\cali{V}_S$ and $R=R_{n,n}$. Then,
\begin{eqnarray*}
d^{-n}\abs{\V_S(R_{n,n})}&\lesssim& d^{-n}\log\left({\eps_{n,n}}^{-4k^2}\right)\\
&=&d^{-n}\cdot[nC(2+N_4)(40k^2\delta)^{6kn}]\cdot 4k^2\cdot(-\log\eps)
\lesssim nd^{-\frac{n}{4}}(-\log\eps).
\end{eqnarray*} The last inequality is from our assumption on $\delta$ in Section ~\ref{sec:settings}.
This proves the lemma.
\end{proof}

\section{Examples and Remarks}\label{sec:exam}
The following example is a case where Theorem ~\ref{thm:mainthm} is applicable but Theorem ~\ref{thm:544} is not. We use a map considered by Forn{\ae}ss-Sibony in ~\cite{criticallyfinitemap}.

\begin{exam}\label{example}
Consider a holomorphic endmorphism $F:\proj^5\to\proj^5$ defined by
\begin{displaymath}
F([z:w:t:\alpha:\beta:\gamma])=[(z-2w)^2:z^2:(z-2t)^2:(\alpha-2\beta)^2:\alpha^2:(\alpha-2\gamma)^2].
\end{displaymath}
The critical set $\cali{C}_F$ of the map $F$ is $\set{z=2w}\cup\set{z=0}\cup\set{z=2t}\cup\set{\alpha=2\beta}\cup\set{\alpha=0}\cup\set{\alpha=2\gamma}$. This map is a post-critically finite map. Let $\cali{P}_F$ denote the post-critical set. We have the following orbits of the set of critical points of $F$:
\begin{eqnarray*}
&&\set{z=0}\to \set{w=0}\to \set{z=w}\circlearrowleft\\
&&\set{z=2w}\to \set{z=0}\to \set{w=0}\to \set{z=w}\circlearrowleft\\
&&\set{z=2t}\to \set{t=0}\to \set{w=t}\rightleftarrows \set{z=t}\\
&&\set{\alpha=0}\to \set{\beta=0}\to \set{\alpha=\beta}\circlearrowleft\\
&&\set{\alpha=2\beta}\to \set{\alpha=0}\to \set{\beta=0}\to \set{\alpha=\beta}\circlearrowleft\\
&&\set{\alpha=2\gamma}\to \set{\gamma=0}\to \set{\beta=\gamma}\rightleftarrows \set{\alpha=\gamma}
\end{eqnarray*}
Thus, $\cali{P}_F=\cali{C}_F\cup\set{w=0}\cup\set{t=0}\cup\set{z=w}\cup\set{w=t}\cup\set{z=t}\cup\set{\beta=0}\cup\set{\gamma=0}\cup\set{\alpha=\beta}\cup\set{\beta=\gamma}\cup\set{\alpha=\gamma}$ is a finite union of projective linear subspaces. We first consider $\mu'_-$. We claim that $\mu'_-\equiv 1$ and the set $E'_\mu$ is empty. The chain rule proves that the set of critical set of $F^n$ is $\cup_{i=0}^{n-1} F^{-i}(\cali{C}_F)$ and therefore, the set of the critical values of $F^n$ is $\cup_{i=1}^n F^i(\cali{C}_F)\subseteq \cali{P}_F$, which is bounded with respect to inclusion independently of $n\in\mathbb{N}$. This implies our claim.

Hence, we only consider $E_\lambda$. We explicitly compute $\mu_-$ for $F$. Note that $\set{[0: 0: t: 0: 0: \gamma]}$ is a complete invariant projective linear subspace for $F$. Therefore the multiplicity of $F^n$ at each point of $\set{[0: 0: t: 0: 0: \gamma]}$ is $16^n$, which means $\mu_-=16$ for the set.

Next, we claim that the forward orbit of each point outside $\set{z=0, w=0}\cup\set{\alpha=0, \beta=0}$ can visit $\cali{C}_F$ at most $4$ times. Observe that $\set{z=2w}$ and $\set{\alpha=2\beta}$ maps into $\set{z=0}$ and $\set{\alpha=0}$ under $F$, respectively. After the first visit to $\set{z=0}\cup\set{z=2t}\cup\set{\alpha=0}\cup\set{\alpha=2\gamma}$ in $\cali{C}_F$, the second visit to $\cali{C}_F$ should take place on the intersection of $\cali{C}_F$ and $\set{w=0}\cup\set{t=0}\cup\set{z=w}\cup\set{w=t}\cup\set{z=t}\cup\set{\beta=0}\cup\set{\gamma=0}\cup\set{\alpha=\beta}\cup\set{\beta=\gamma}\cup\set{\alpha=\gamma}$.
After the second visit to $\cali{C}_F$, the images of those intersection points afterwards are completely determined as follows:
\begin{eqnarray*}
&&\set{z=w=0}\circlearrowleft\\
&&\set{z=t=0}\to \set{w=t=0}\to \set{z=w=t}\circlearrowleft\\
&&\set{z=0, w=t}\to \set{z=t, w=0}\to \set{z=w=t}\circlearrowleft\\
&&\set{z=2w, t=0}\to \set{z=0, w=t}\to \set{z=t, w=0}\to \set{z=w=t}\circlearrowleft\\
&&\set{z=2w=2t}\to \set{z=t=0}\to \set{w=t=0}\to \set{z=w=t}\circlearrowleft\\
&&\set{z=2w=t}\to \set{z=0, w=t}\to \set{z=t, w=0}\to \set{z=w=t}\circlearrowleft\\
&&\set{z=2t, w=0}\to \set{z=w=4t}\circlearrowleft\\
&&\set{z=w=2t}\to \set{z=w, t=0}\to \set{z=w=t}\circlearrowleft
\end{eqnarray*}
and the same is true when $z,w,t$ are replaced by $\alpha, \beta, \gamma$.
These cases prove our claim. Therefore, the multiplicity of $F^n$ at a point outside $\set{z=0, w=0}\cup\set{\alpha=0, \beta=0}$ is uniformly bounded with respect to the choice of the point and $n\in\mathbb{N}$. This implies that $\mu_-=1$ for those points. In the same way, we obtain $\mu_-=4$ for the remaining case. Summarizing it, we have
\begin{displaymath}
\mu_-=\left\lbrace
\begin{array}{l}
16\textrm{ for }\set{[0: 0: t: 0: 0: \gamma]}\\
4\textrm{ for }[\set{z=0, w=0}\cup\set{\alpha=0, \beta=0}]\setminus\set{z=w=\alpha=\beta=0}\\
1\textrm{ elsewhere}.
\end{array}\right.
\end{displaymath}
It is not too difficult to see that $E=\set{z=0, w=0}\cup\set{\alpha=0, \beta=0}$. Consider the current of integration over the projective linear subspace defined by $\set{z=w=3t, \alpha=\beta=3\gamma}$. We normalize it by its mass and call the normalized one $S$. Note that $S$ is a positive closed $(4, 4)$-current. Theorem ~\ref{thm:mainthm} is applicable to $S$ in order to verify the convergence towards $T^4$. However, on $E$, the hypothesis of Lemma 5.4.5 in \cite{superpotentials} is violated. Indeed, it is clear that there is no $N\in\mathbb{N}$ such that $(20\cdot 5^2 4^N)^{8\cdot 5}<32^N$. Thus, we cannot use Theorem ~\ref{thm:544} to $S$. 
\end{exam}

\begin{rem}
By applying the same proof as in Theorem \ref{thm:mainthm} with minor changes, we can show that the speed of convergence in Corollary 5.5.6 in \cite{superpotentials} is exponential. Indeed, the hyperplane at infinity becomes $\Phi_1$ and the set $I_-$ becomes $E$.
\end{rem}

\begin{rem}
The convergence in Theorem \ref{thm:mainthm} is not uniform with respect to $S\in\cali{C}_p$ in general. It is rather pointwise convergence for $S\in\cali{C}_p$ smooth on $E$. Indeed, in the estimate in Section \ref{sec:2}, the sup-norm of $S$ near $E$ comes into the computations. However, if $f$ has as its $E$ the empty set, then the estimate in Section \ref{sec:2} does not appear. Then, the mass of $S$ only matters and the proof becomes the same as in Theorem 5.4.4 in \cite{superpotentials}. It means that for $f$ with the empty set as its $E$, we have uniform convergence with respect to $S\in\cali{C}_p$. The author would like to thank the referee for his comments on this uniformity issue.
\end{rem}
\newpage

\begin{thebibliography}{10}

\bibitem{BD}
Jean-Yves Briend and Julien Duval.
\newblock Deux caract\'{e}risations de la measure d'equilibre d'un
  endomorphisme de $p^k(\cplx)$.
\newblock {\em Publ. Math. Inst. Hautes \'{E}tudes Sci.}, 93:145--159, 2001.

\bibitem{HansBrolin}
Hans Brolin.
\newblock Invariant sets under iteration of rational functions.
\newblock {\em Ark. Mat.}, 6:103--144, 1965.

\bibitem{Demailly}
Jean-Pierre Demailly.
\newblock Complex analytic geometry.
\newblock available at www.fourier.ujf-grenoble.fr/$\sim$demailly.

\bibitem{analyticmulticocycle}
Tien-Cuong Dinh.
\newblock Analytic multiplicative cocycles over holomorphic dynamical systems.
\newblock {\em Complex Var. Elliptic Equ.}, 54:243--251, 2009.

\bibitem{DSHreference}
Tien-Cuong Dinh and Nessim Sibony.
\newblock Distribution des valeurs de transformations m\'{e}romorphes et
  applications.
\newblock {\em Comment. Math. Helv.}, 81:221--258, 2006.

\bibitem{pullbackbyholo}
Tien-Cuong Dinh and Nessim Sibony.
\newblock Pull-back of currents by holomorphic maps.
\newblock {\em Manuscripta Math.}, 123:357--371, 2007.

\bibitem{equidist_holo}
Tien-Cuong Dinh and Nessim Sibony.
\newblock Equidistribution towards the {G}reen current for holomorphic maps.
\newblock {\em Ann. Sci. \'{E}cole Norm. Sup.}, 41:307--336, 2008.

\bibitem{superpotentials}
Tien-Cuong Dinh and Nessim Sibony.
\newblock Super-potentials of currents, intersection theory and dynamics.
\newblock {\em Acta Math.}, 203:1--82, 2009.

\bibitem{dynamicsinscv}
Tien-Cuong Dinh and Nessim Sibony.
\newblock Dynamics in several complex variables: Endomorphisms of projective
  spaces and polynomial-like mappings.
\newblock In {\em Holomorphic Dynamical Systems, Cetraro, Italy, July 7-12,
  2008}, Lecture Notes in Mathematics, pages 165--294. Springer-Verlag Berlin
  Heidelberg, 2010.

\bibitem{equidistributionspeed}
Tien-Cuong Dinh and Nessim Sibony.
\newblock Equidistribution speed for endomorphisms of projective spaces.
\newblock {\em Math. Ann.}, 347:613--626, 2010.

\bibitem{favre}
Charles Favre.
\newblock {N}ote on pull-back and {L}elong number of currents.
\newblock {\em Bull. Soc. Math. France}, 127(3):445--458, 1999.

\bibitem{favre0}
Charles Favre.
\newblock {\em Dynamique des applications rationelles}.
\newblock PhD thesis, Universit\'{e} de Paris-Sud, Orsay, 2000.

\bibitem{favre1}
Charles Favre.
\newblock Multiplicity of holomorphic functions.
\newblock {\em Math. Ann.}, 316(2):355--378, 2000.

\bibitem{brolinthm}
Charles Favre and Mattias Jonsson.
\newblock {B}rolin's theorem for curves in two complex dimensions.
\newblock {\em Ann. Inst. Fourier (Grenoble)}, 53(5):1461--1501, 2003.

\bibitem{eigenvaluations}
Charles Favre and Mattias Jonsson.
\newblock Eigenvaluations.
\newblock {\em Ann. Sci. \'{E}cole Norm. Sup.}, 40(2):309--349, 2007.

\bibitem{federer}
Herbert Federer.
\newblock {\em Geometric Measure Theory}.
\newblock Springer-Verlag New York Inc., 1969.

\bibitem{criticallyfinitemap}
John~Erik Forn{\ae}ss and Nessim Sibony.
\newblock Critically finite rational maps on $\proj^2$.
\newblock {\em Contemporary Mathematics}, 137:245--260, 1992.

\bibitem{cplxpotentialtheory}
John~Erik Forn{\ae}ss and Nessim Sibony.
\newblock Complx dynamics in higher dimensions.
\newblock In {\em Complex potential theory (Montreal, PQ, 1993)}, volume 439,
  pages 131--186. Kluwer Acad. Publ., Dordrecht, 1994.

\bibitem{cplxdynII}
John~Erik Forn{\ae}ss and Nessim Sibony.
\newblock Complex dynamics in higher dimension {II}.
\newblock In {\em Modern methods in complex analysis (Princeton, NJ, 1992)},
  volume 137 of {\em Ann. Math. Stud.}, pages 135--182. Princeton University
  Press, 1995.

\bibitem{FLM}
Alexandre Freire, Artur Lopes, and Ricardo Ma\~{n}\'{e}.
\newblock An invariant measure for rational maps.
\newblock {\em Bol. Soc. Brasil. Mat.}, 14(1):45--62, •.

\bibitem{Gignac}
William Gignac.
\newblock Measures and dynamics on {N}oetherian spaces.
\newblock to appear in the Journal of Geometric Analysis, arXiv:1202.0793.

\bibitem{Guedj}
Vincent Guedj.
\newblock Equidistribution towards the {G}reen current.
\newblock {\em Bull. Soc. Math. France}, 131(3):359--372, 2003.

\bibitem{cplxanalvar}
Robert~C. Gunning.
\newblock {\em Lectures on Complex Analytic Varieties: the Local
  Parametrization Theorem}.
\newblock Preliminary Informal Notes of University Courses and Seminars in
  Mathematics. Princeton University Press, 1970.

\bibitem{lelong}
Pierre Lelong.
\newblock {\em Fonctions Plurisousharmoniques et Formes Diff\'{e}rentielles
  Positives}.
\newblock Dunod, 1968.

\bibitem{Lyubich}
Mikhail Ljubich.
\newblock Entropy properties of rational endomorphisms of the {R}iemann sphere.
\newblock {\em Ergodic Theory Dynam. Systems}, 3(3):351--385, 1983.

\bibitem{idealsofdiffftns}
Bernard Malgrange.
\newblock {\em Ideals of Differentiable Functions}.
\newblock Oxford University Press, 1966.

\bibitem{ueda}
S.~Morosawa, Y.~Nishimura, M.~Taniguchi, and T.~Ueda.
\newblock {\em Holomorphic Dynamics}.
\newblock Cambridge University Press, 2000.

\bibitem{Parra1}
Rodrigo Parra.
\newblock The {J}acobian cocycle and equidistribution towards the {G}reen
  current.
\newblock arXiv:1103.4633.

\bibitem{Parra}
Rodrigo Parra.
\newblock {\em Currents and equidistribution in holomorphic dynamics}.
\newblock PhD thesis, Univ. of Michigan, 2011.

\bibitem{RS}
Alexander Russakovskii and Bernard Shiffman.
\newblock Value distribution for sequences of rational mappings and complex
  dynamics.
\newblock {\em Indiana Univ. Math. J.}, 46(3):897--932, 1997.

\bibitem{sibony}
Nessim Sibony.
\newblock Dynamique des applications rationnelles de $\proj^k$.
\newblock {\em Panoramas $\&$ Synth\`{e}ses}, 8:97--185, 1999.

\bibitem{Taflin}
Johan Taflin.
\newblock Equidistribution speed towards the {G}reen current for endomorphisms
  of $\proj^k$.
\newblock {\em Adv. Math.}, 227(5):2059--2081, 2011.

\end{thebibliography}

\end{document}